\documentclass[12pt]{amsart}
\usepackage{amssymb,latexsym, graphicx}
 \usepackage[all]{xy}
\usepackage{array}
\newtheorem{theorem}{Theorem}[section]
\newtheorem{lemma}[theorem]{Lemma}
\newtheorem{proposition}[theorem]{Proposition}
\newtheorem{corollary}[theorem]{Corollary}

\newtheorem*{theorem1point3}{Theorem 1.3}

\newtheorem*{brumerconjecture}{Brumer Conjecture}
\newtheorem*{BSconj}{Brumer-Stark Conjecture}

\newtheorem*{symbolproperties}{Properties of $(\sigma, \cdot)_{\Kum}$ and $(\sigma, \cdot)_{\N}$}
\newtheorem*{intproperty}{Integrality Property}

\theoremstyle{definition}
\newtheorem*{definition}{Definition}

\theoremstyle{remark}
\newtheorem*{remark}{Remark}

\newtheorem*{claim}{Claim}

\def\noqed{\renewcommand{\qedsymbol}{}}

\DeclareMathOperator{\rk}{\mathrm{rk}}
\DeclareMathOperator{\N}{\mathrm{N}}
\DeclareMathOperator{\Ann}{\mathrm{Ann}}

\newcommand{\Cl}{\mathrm{Cl}}

\newcommand{\Gal}{\mathrm{Gal}}

\newcommand{\Hom}{\mathrm{Hom}}

\newcommand{\ns}{\mathrm{ns}}
\newcommand{\Fit}{\mathrm{Fit}}
\newcommand{\Kum}{\mathrm{Kum}}

\begin{document}
\title{Divisibility of partial zeta function values at zero for degree $2p$ extensions}
\author{Barry~R. Smith}

\begin{abstract}
Let $K/k$ be an Abelian extension of number fields, $S$ be a set of places of $k$, and $p$ be an odd prime number.  We continue an investigation begun in \cite{bS2012a} into the values at 0 of the $S$-imprimitive partial zeta functions of $K/k$.  The main result of \cite{bS2012a} provides, under the
assumption that the $p$-power roots of unity in $K$ are cohomologically trivial, a criterion for the values to have larger than expected $p$-valuation.  The present paper provides such a criterion for a special class of degree $2p$ extensions for which the $p$-power
roots of unity are not cohomologically trivial.  For such extensions, new sufficient conditions are also given for the $p$-local Brumer-Stark conjecture for $K/k$ and for Leopoldt's conjecture on the number of independent $\mathbb{Z}_p$-extensions of $k$.
\end{abstract}

\maketitle

\section{Introduction}\label{S:intro}

In this article is set forth a new arithmetical interpretation of the values of partial zeta functions at $0$ for a special class of Abelian number field extensions.  

Let $K/k$ be an Abelian  extension of number fields.  A partial zeta function $\zeta_{K/k, S} (\sigma, s)$ is associated with each automorphism $\sigma$ of $K/k$ and finite set $S$ of places of $k$. Siegel \cite{cS1970} and Shintani \cite{tS1976} independently proved that the value $\zeta_{K/k, S} (\sigma, -n)$ is rational when $n$ is a nonnegative integer and $S$ contains the Archimedean places and the places that ramify in $K/k$.  Extending the work of Siegel and Shintani respectively, Deligne and Ribet \cite{DR1980} and Cassou-Nogu\`{e}s \cite{pC1979} bounded the denominator of $\zeta_{K/k, S} (\sigma, -n)$, proving in particular that $w_K \zeta_{K/k, S} (\sigma, 0)$ is an integer, where $w_K$ is the number of roots of unity in $K$.

The Brumer-Stark Conjecture and the Coates-Sinnott Conjecture provide arithmetical interpretations of the values $\zeta_{K/k, S} (\sigma, -n)$.  They state that when the partial zeta functions of the automorphisms of $K/k$ are assembled into one equivariant $L$-function, its values at non-positive integers are related to the Galois module structure of the class group and higher \'{e}tale cohomology groups of $K$.

In previous work, the author investigated instead the factors of the individual integers $w_K \zeta_{K/k, S} (\sigma, 0)$.  The main result of \cite{bS2012a} states that when $p^r$ is an odd prime power factor of $w_K$, the integers $w_K \zeta_{K/k, S} (\sigma, 0)$ are divisible by $p^r$ if the group of $p$-power roots of unity in $K$ is $\Gal (K/k)$-cohomologically trivial and $\Gal(K/k)$-isomorphic to a submodule of the class group of $K$.  It is contingent on Brumer's Conjecture for the extension $K/k$.

The present article reveals a new interpretation of the integers $w_K \zeta_{K/k, S} (\sigma, 0)$ for a special class of degree $2p$ Abelian number field extensions whose group of $p$-power roots of unity is not $\Gal(K/k)$-cohomologically trivial:

\begin{definition}
Fix an odd prime number $p$ and let $K_1/k_0$ be a degree $2p$ Abelian extension of number fields with $k_0$ totally real and $K_1$ totally complex.  Let $K_0$ and $k_1$ be the unique intermediate fields with $\left[ K_0 \colon k_0 \right] = 2$ and $\left[ k_1 \colon k_0 \right] = p$.  Let $N \colon \left( \Cl_{K_1} \otimes \mathbb{Z}_p \right)^- \rightarrow \left( \Cl_{{K_0}} \otimes \mathbb{Z}_p \right)^-$ be the norm map between the components of the $p$-primary class groups of $K_1$ and $K_0$ upon which complex conjugation acts by inversion.  We say $K_1/k_0$ is a \textbf{TK extension} if
\begin{enumerate}
	\item The group of $p$-power roots of unity in $K_1$ is not $\Gal(K_1/k_0)$-cohomologically trivial
	\item No prime ideal of $k_0$ splits in $K_0/k_0$ and ramifies in $k_1/k_0$
	\item The norm map $N \colon \left( \Cl_{K_1} \otimes \mathbb{Z}_p \right)^- \rightarrow \left( \Cl_{K_0} \otimes \mathbb{Z}_p \right)^-$ has trivial kernel
\end{enumerate}
If in addition, no prime of $k_0$ dividing $p$ splits in $K_0/k_0$, we say $K_1/k_0$ is a \textbf{TKNS extension}.
\end{definition}

Proposition \ref{P:trivialproperties} will show that TK extensions are unramified away from $p$. Thus, the extension $k_1/k_0$ is an Abelian extension of $k_0$ unramified away from $p$ and different from the cyclotomic one.  One of our main results (Corollary \ref{C:nosplitconditions}) states that when $K_1/k_0$ is a TKNS extension and the number of $p$-power roots of unity in $K_1$ is $p^r$, then $p^r$ divides each integer $w_{K_1} \zeta_{K_1/k_0, S} (\sigma, 0)$ if and only if there exists a cyclic extension of $k_0$ of degree $p^{r+1}$, unramified away from $p$ and linearly disjoint over $k_0$ from the cyclotomic $\mathbb{Z}_p$-extension.  It is contingent on the Brumer-Stark Conjecture for $K_1/k_0$.  In fact, the main stepping-stone in the proof is a reduction of the $p$-local Brumer-Stark conjecture for TK extensions to the vanishing of a certain cohomology class (Theorem \ref{T:bsharprank1}).

Notwithstanding their special nature, it is easy to produce examples of TKNS extensions.  To concretize Corollary \ref{C:nosplitconditions}, the reader is invited to look at the table of examples on page 25 computed using the PARI/GP software.  Each row represents an Abelian extension $K_1/k_0$ of degree 6.  The group of roots of unity in each $K_1$ has order 6.  A ``yes'' in the $\mathfrak{K}_p$ column indicates the extension is of type TKNS and so satisfies the hypotheses of our main results, Theorem \ref{T:main} and Corollary \ref{C:nosplitconditions}.  The penultimate column contains an element of $\mathbb{Z}[\Gal(K_1/k_0)]$ whose coefficients are the values $6 \zeta_{K_1/k_0, S} (\sigma, 0)$.  The last column displays ``yes'' if and only if there exists a cyclic extension of $k_0$ of degree $9$, unramified away from $3$ and linearly disjoint over $k_0$ from the cyclotomic $\mathbb{Z}_3$-extension.  The reader should observe that the coefficients in the penultimate column are multiples of $3$ if and only if the final column says ``yes''.  This exemplifies the new interpretation of partial zeta function values being set forth in this article.  While it is more specialized that given in \cite{bS2012a}, it lies considerably deeper.



We now elaborate. Let $K/k$ be an Abelian extension of number fields with $k$ totally real and $K$ totally complex.  Denote the Galois group of $K/k$ by $G$.  Let $S$ be a finite set of places of $k$ containing the Archimedean places and the prime ideals that ramify in $K/k$. 
If $\mathfrak{p}$ is a prime ideal of $k$ not contained in $S$, we denote the Frobenius automorphism in $G$ associated with $\mathfrak{p}$ by $\sigma_{\mathfrak{p}}$.  For each $\chi$ in the group $\widehat{G}$ of complex-valued characters of $G$, the Abelian $L$-function deprived of the Euler factors corresponding to primes in $S$ is defined by
\begin{equation*}
	L_{K/k, S} (s, \chi) = \prod_{\mathfrak{p} \not\in S} \left( 1 - \frac{\chi 
	\left( \sigma_{\mathfrak{p}} \right)}{N \mathfrak{p}^{s}} \right)^{-1}.
\end{equation*}
This function converges absolutely and uniformly on compact subsets of $\mathrm{Re} (s) > 1$.  It can be analytically continued to an entire function excepting a simple pole at $s=1$ when $\chi$ is the trivial character.

\begin{definition}
The $S$-imprimitive \textbf{equivariant $L$-function} is the function\\ $\theta_{K/k, S} \colon \mathbb{C} \setminus \{ \, 1 \, \} \rightarrow \mathbb{C}[G]$ defined by
\begin{equation*}
	\theta_{K/k, S} (s) = \sum_{\chi \in \widehat{G}} L_{K/k, S} \left( s, \overline{\chi} \right) e_{\chi},
\end{equation*}
in which $e_{\chi}$ is the idempotent corresponding to $\chi$. 
\end{definition}
\noindent When $k$ and $K$ are the same field, $\theta_{K/k, S}$ is just the $S$-imprimitive Dedekind zeta function of $k$.

\begin{definition}
The $S$-imprimitive \textbf{partial zeta functions} $\zeta_{K/k, S} (\sigma, s)$ are defined through the relationship
\begin{equation*}
	\theta_{K/k, S} (s) = \sum_{\sigma \in G} \zeta_{K/k, S} (\sigma, s) \sigma^{-1}
\end{equation*}
\end{definition}
The work of Siegel \cite{cS1970} and Shintani \cite{tS1976} mentioned above shows that $\theta_{K/k,S} (0)$ has rational coefficients.  Let $\mu_{K}$ and $\Cl_K$ denote respectively the group of roots of unity in $K$ and the class group of $K$, and let $\Ann_{\mathbb{Z}[G]} \mu_K$ and $\Ann_{\mathbb{Z}[G]} \Cl_K$ be their $\mathbb{Z}[G]$-annihilators.  Deligne and Ribet \cite{DR1980} and Cassou-Nogu\`{e}s \cite{pC1979} obtained bounds on the denominators of the rational numbers $\zeta_{K/k, S} (\sigma, s)$ as a consequence of the following property of $\theta_{K/k, S} (0)$:
\begin{intproperty}
	$\Ann_{\mathbb{Z}[G]} \mu_K \, \cdot \, \theta_{K/k, S} (0) \subseteq \mathbb{Z}[G]$
\end{intproperty}
\noindent The Brumer Conjecture refines this while generalizing the analytic class number formula for Dedekind zeta functions:
\begin{brumerconjecture}
	$\Ann_{\mathbb{Z}[G]} \mu_K \, \cdot \, \theta_{K/k, S} (0) \subseteq \Ann_{\mathbb{Z}[G]} \Cl_K$
\end{brumerconjecture}

The following proposition is an observation of David Hayes (\cite{dHwp}, Corollary 3.2).  Because \cite{dHwp} is unpublished, the proof is reproduced here.
\begin{proposition}\label{P:samedenominators}
When written in lowest terms, all coefficients of $\theta_{K/k, S} (0)$ have the same denominator, a divisor of $w_K$.
\end{proposition}
\begin{proof}
Write $w_K \theta_{K/k,S} (0) = \sum_{\sigma \in G} a_{\sigma} \sigma^{-1}$.  By the integrality property, the coefficients $a_{\sigma}$ are integers. For each $\sigma$ in $G$, let $N \sigma$ be an integer such that $\zeta^{\sigma} = \zeta^{N \sigma}$ for all roots of unity $\zeta$ in $K$.  The key observation is the congruence
\begin{equation*}
	a_{\tau \sigma} \equiv N \sigma \cdot a_{\tau} \pmod{ w_K},
\end{equation*}
which holds for all $\tau$, $\sigma$ in $G$.   Since $\sigma - N \sigma $ is in $\Ann_{\mathbb{Z}[G]} \mu_K$, the integrality property shows that $(\sigma - N \sigma) w_K \theta_{K/k, S} (0)$ is in $w_K \mathbb{Z}[G]$.  The congruence follows by considering the coefficient of $\tau^{-1}$ in $w_K \theta_{K/k, S} (0)$.  Because $N \sigma$ and $w_K$ are relatively prime, it follows that $\mathrm{gcd} (a_{\tau \sigma}, w_K) = \mathrm{gcd} (a_{\tau}, w_K)$ for all $\tau$, $\sigma$ in $G$.
\end{proof}

David Hayes conjectured that the denominators of the coefficients of $\theta_{K/k, S} (0)$ are influenced by the arithmetic of $K$.  His investigations \cite{dH2004}, \cite{dHman} inspired the formulation and proof of the following theorem \cite{bS2012a}:
\begin{theorem}\label{T:priortheorem}
Let $K/k$ be an Abelian extension with Galois group $G$ satisfying the Brumer Conjecture.  Choose an odd prime number $l$, and assume that the group $\mu_K \otimes \mathbb{Z}_{l}$ of $l$-power roots of unity in $K$ is $G$-cohomologically trivial and $G$-isomorphic with a submodule or quotient of $\Cl_K$.  Set $l^r = \left| \mu_K \otimes \mathbb{Z}_l \right|$.  Then
\begin{equation*}
	w_K \theta_{K/k, S} (0) \in l^r \mathbb{Z}[G],
\end{equation*}

\end{theorem}

Unless $k$ is totally real and $K$ is totally complex, $\theta_{K/k, S} (0)$ is zero and the theorem is trivial.  This article is a case study of the smallest totally complex extensions of totally real fields for which $\mu_{K} \otimes \mathbb{Z}_l$ is not cohomologically trivial ---  those $K/k$ of degree $2p$ such that $\left[ k (\zeta_{p^r}) \colon k \right] = 2$ in which $\zeta_{p^r}$ is a generator of the group of $p$-power roots of unity in $K$.  We will see that the divisibility properties of $w_K \theta_{K/k, S} (0)$ for most such extensions are regulated by the arithmetic of the extension $K/k$ in a simple way (Theorems \ref{T:simplecase} and \ref{T:exceptionalsimplecase}).  But for the TK extensions, defined above, the regulation is much more sophisticated.

Our result  is contingent on the $p$-local Brumer-Stark conjecture, which we now state. We call an element $\alpha$ of $K^{\times}$ an anti-unit if $\left| \phi(\alpha) \right|=1$ for each embedding ${\phi \colon K \hookrightarrow \mathbb{C}}$. If $n$ is a divisor of $w_K$, we call an element $\alpha$ in $K$ $n$-Abelian (for $K/k$) if the extension $K\left(\sqrt[n]{\alpha}\right)/k$ is Abelian. 
\begin{BSconj}
Let $\theta = \theta_{K/k, S}(0)$ be the value at $0$ of the $S$-imprimitive equivariant $L$-function of an Abelian extension $K/k$ of number fields.  If $\mathfrak{a}$ is an integral ideal in $K$, then there exists an element $\varepsilon_{\mathfrak{a}, K/k, S}$ in $K^{\times}$ satisfying:
\begin{enumerate}
	\item $\mathfrak{a}^{w_K \theta} = \left( \varepsilon_{\mathfrak{a}, K/k, S} \right)$
	\item $\varepsilon_{\mathfrak{a}, K/k, S}$ is an anti-unit
	\item $\varepsilon_{\mathfrak{a}, K/k, S}$ is $w_K$-Abelian.
\end{enumerate}
\end{BSconj}

The Brumer and Brumer-Stark conjectures have local statements for each prime number $l$, and each is equivalent to the aggregate of its local conjectures.  Let $l$ be a prime number and set $l^r = \left| \mu_K \otimes \mathbb{Z}_l \right|$.  The $l$-local Brumer-Stark conjecture is obtained by making the following modifications to the Brumer-Stark conjecture:  the ideal $\mathfrak{a}$ is chosen to represent a class in $\Cl_K \otimes \mathbb{Z}_l$, and we only require $\varepsilon_{\mathfrak{a}, K/k, S}$ to be $l^r$-Abelian.

For odd primes $l$, Popescu proved the $l$-local Brumer-Stark and $l$-local Brumer conjectures are equivalent when the $l$-power roots of unity in $K$ are cohomologically trivial.  It is therefore natural that an attempt to generalize Theorem \ref{T:priortheorem} to extensions whose roots of unity have nontrivial cohomology requires the Brumer-Stark conjecture.

We now specialize to degree $2p$ extensions and define new notation.  Let $p$ be an odd prime number and fix a cyclic extension $K_1/k_0$ of degree $2p$ with $k_0$ totally real and $K_1$ totally complex.  The unique intermediate fields $K_0$ and $k_1$ between $k_0$ and $K_1$ with $\left[ K_0 \colon k_0 \right] = 2$ and $\left[ k_1 \colon k_0 \right] = p$ are respectively totally complex and totally real.  Denote the Galois group of $K_1/k_0$ by $G$. Let $S^{\mathrm{min}}$ be the set of places of $k_0$ consisting of the Archimedean places and the prime ideals that ramify in $K_1/k_0$. Let $\mu_{K_0}$ and $\mu_{K_1}$ be the groups of roots of unity in $K_0$ and $K_1$, and set $w_0 = \left| \mu_{K_0} \right|$ and $w_1 = \left| \mu_{K_1} \right|$.

If $K$ is one of the fields mentioned above, let $S_{K}^{\ns}$ be the set of places in $K$ above those in $S^{\mathrm{min}}$ that do not split in $K_{0}/k_{0}$. Let $\mathfrak{C}_{0}$ and $\mathfrak{C}_{1}$ denote the cokernels of the canonical maps of $S$ class groups $\Cl_{k_{0}, S_{k_{0}}^{\ns}} \rightarrow \Cl_{K_{0}, S_{K_{0}}^{\ns}}$ and $\Cl_{k_{1}, S_{k_{1}}^{\ns}} \rightarrow \Cl_{K_{1}, S_{K_{1}}^{\ns}}$.  Proposition 2.2 in \cite{bS2012b} shows that the norm map $N \colon \mathfrak{C}_1 \rightarrow \mathfrak{C}_0$ is surjective.
We denote its kernel by $\mathfrak{K}$.  For simplicity, we also write
\begin{equation*}
	\mathfrak{C}_{0,p} = \mathfrak{C}_0 \otimes \mathbb{Z}_p, \quad
	\mathfrak{C}_{1,p} = \mathfrak{C}_1 \otimes \mathbb{Z}_p, \quad
	\mathfrak{K}_p = \mathfrak{K} \otimes \mathbb{Z}_p
\end{equation*}
Lemma 2.1 in \cite{bS2012b} shows that for $i=0$ or $1$, $\mathfrak{C}_{i,p}$ is isomorphic to $\left( \Cl_{K_i} \otimes \mathbb{Z}_p \right)^-$, the component of the $p$-primary part of the class group of $K_i$ upon which complex conjugation acts by inversion.  Further, one can deduce that through these isomorphisms that $\mathfrak{K}_p$ is isomorphic to the kernel of the norm map $N \colon \left( \Cl_{K_1} \otimes \mathbb{Z}_p \right)^- \rightarrow \left( \Cl_{K_0} \otimes \mathbb{Z}_p \right)^-$.  Thus, for TK extensions, $\mathfrak{K}_p$ is trivial.

Recall that $\varepsilon_{\mathfrak{a}, K_1/k_1, S_{k_1}^{\mathrm{ns}}}$ denotes the element corresponding to the ideal $\mathfrak{a}$ through the Brumer-Stark conjecture for the quadratic extension $K_1/k_1$ and $S_{k_1}^{\mathrm{ns}}$. Our main result is:\\[0.1cm]

\begin{theorem}\label{T:main}
Let $K_1/k_0$ be a TK extension for which the $p$-local Brumer-Stark Conjecture holds.  Set $\left| \mu_{K_1} \otimes \mathbb{Z}_p \right| = p^r$.  The following are equivalent:
\begin{enumerate}
	\item $w_{1} \theta_{K_1/k_0, S^{\mathrm{min}}} (0)$ is in $p^r \mathbb{Z}[G]$
	\item As $\mathfrak{a}$ runs through classes in $\left( \Cl_{K_1} \otimes \mathbb{Z}_p \right)^{-}$, the elements $\varepsilon_{\mathfrak{a}, K_1/k_1, S_{k_1}^{\mathrm{ns}}}$ can be chosen so that their $p^r$th roots generate a cyclic extension of $k_0$ of degree $2p^{r+1}$, unramified away from $p$ and linearly disjoint over $k_0$ from the cyclotomic $\mathbb{Z}_p$-extension.
\end{enumerate} 
For conditions (i) and (ii) to hold, it is necessary that there exist $G$-isomorphisms
\begin{equation*}
	\left( \Cl_{K_0} \otimes \mathbb{Z}_p \right)^- \cong \left( \Cl_{K_1} \otimes \mathbb{Z}_p \right)^- \cong \mu_{K_1} \otimes \mathbb{Z}_p
\end{equation*}
\end{theorem}

\begin{remark}
It is easier to show the equivalence of condition (i) with condition (iii): the elements $\varepsilon_{\mathfrak{a}, K_1/k_0, S^{\mathrm{min}}}$ can be chosen so that their $p^r$th roots generate a cyclic extension of $k_0$ of degree $2p^{r+1}$, unramified away from $p$ and linearly disjoint over $k_0$ from the cyclotomic $\mathbb{Z}_p$-extension (see Theorem \ref{T:maintheoremalternate}).  Condition (ii) states instead that $\varepsilon_{\mathfrak{a}, K_1/k_1, S_{k_1}^{\mathrm{ns}}}$, which the Brumer-Stark conjecture says is $p^r$-Abelian for $K_1/k_1$, is actually $p^r$-Abelian for $K_1/k_0$.
\end{remark}

When $K_1/k_0$ is a TKNS extension, we will show in Corollary \ref{C:nosplitconditions} that the conditions of Theorem \ref{T:main}  are equivalent to condition (iv): there exists a cyclic extension $L/k_0$ of degree $p^{r+1}$, unramified away from $p$ and linearly disjoint over $k_0$ from the cyclotomic $\mathbb{Z}_p$ extension. In addition, $L \left( \zeta_{p^{2r+1}} \right)^+$ is then the maximal Abelian exponent $p^{r+1}$ extension of $k_0$ unramified away from $p$.  Finally, if $K_1/k_0$ is a TKNS extension and $\left( \Cl_{K_0} \otimes \mathbb{Z}_p \right)^- \not\cong \mu_{K_0} \otimes \mathbb{Z}_p$, we prove in Corollary \ref{C:Leopoldt} the $p$-primary Leopoldt's conjecture for $k_0$: the only $\mathbb{Z}_p$ extension of $k_0$ is the cyclotomic one.

Theorem \ref{T:main} is a consequence of the equivalence of the $p$-local Brumer-Stark conjecture with the vanishing of a cohomology class in $H^1 \left( \Gal(K_1/K_0),\mu_{K_1} \otimes \mathbb{Z}_p \right)$ whose definition involves the partial zeta function values $\zeta_{K_1/k_0, S^{\mathrm{min}}} (\sigma, 0)$ (see Theorem \ref{T:bsharprank1}).     

Examples of TK and TKNS extensions and computations illustrating the main results for these examples are provided in Section \ref{S:examples}.

\section{The simple case: non-TK extensions}\label{S:simplecase}

The remainder of this article makes frequent use of results from Sections 2 and 3 of \cite{bS2012b}.  The reader is invited to read these sections now.

We recall again the notation introduced in Section \ref{S:intro}.  Fix an odd prime number $p$ and a cyclic extension $K_1/k_0$ of degree $2p$ with $k_0$ totally real and $K_1$ totally complex. The unique intermediate fields $K_0$ and $k_1$ between $k_0$ and $K_1$ with $\left[ K_0 \colon k_0 \right] = 2$ and $\left[ k_1 \colon k_1 \right] = p$ are respectively totally complex and totally real.  Denote the Galois group of $K_1/k_0$ by $G$.  Let $S^{\mathrm{min}}$ be the set of places of $k_0$ consisting of the Archimedean places and the prime ideals that ramify in $K_1/k_0$. Let $\mu_{K_0}$ and $\mu_{K_1}$ be the groups of roots of unity in $K_0$ and $K_1$. 

If $K$ is one of the fields mentioned above, let $S_{K}^{\ns}$ be the set of places in $K$ above those in $S^{\mathrm{min}}$ that do not split in $K_{0}/k_{0}$. Let $\mathfrak{C}_{0}$ and $\mathfrak{C}_{1}$ denote the cokernels of the canonical maps of $S$ class groups $\Cl_{k_{0}, S_{k_{0}}^{\ns}} \rightarrow \Cl_{K_{0}, S_{K_{0}}^{\ns}}$ and $\Cl_{k_{1}, S_{k_{1}}^{\ns}} \rightarrow \Cl_{K_{1}, S_{K_{1}}^{\ns}}$.  
We let $\mathfrak{K}$ denote the kernel of the norm map $N \colon \mathfrak{C}_1 \rightarrow \mathfrak{C}_0$.  For simplicity, we also write $\mathfrak{C}_{0,p} = \mathfrak{C}_0 \otimes \mathbb{Z}_p$, $\mathfrak{C}_{1,p} = \mathfrak{C}_1 \otimes \mathbb{Z}_p$, and $\mathfrak{K}_p = \mathfrak{K} \otimes \mathbb{Z}_p$. 

We introduce new notation as well. Let $\tau$ and $\sigma$ be elements of orders $2$ and $p$ in $G$ (so that $\tau$ is complex conjugation).  Let $N_H$ be the norm element in $\mathbb{Z}[G]$ corresponding to $H = \Gal(K_1/K_0)$. Let $\chi$ be a generator of the group of character group $\widehat{G}$.   For $i=0$ or $1$, we write $\left( \Cl_{K_i} \otimes \mathbb{Z}_p \right)^-$ for the component of $\Cl_{K_i}$ upon which $\tau$ acts by inversion.

Set $w_0 = \left| \mu_{K_0} \right|$, $w_1 = \left| \mu_{K_1} \right|$, and  $p^r = \left| \mu_{K_1} \otimes \mathbb{Z}_p \right|$. Let $\zeta_{p^r}$ be a generator of the group of $p$-power roots of unity in $K_1$. Let $q = \frac{w_1}{w_0}$, and note that the exact power of $p$ dividing $q$ is either $p^0$ or $p^1$.  If $p$ divides $w_0$, then the power is $p^1$ precisely when the $p$-power roots of unity in $K_1$ are $G$-cohomologically trivial.  That is, the group of $p$-power roots of unity in $K_1$ is not $G$-cohomologically trivial if and only if $K_0 = k_0 (\zeta_{p^r})$.

Replacing $S^{\mathrm{min}}$ with a larger set of places only produces weaker results.  Therefore, we will only consider $S^{\mathrm{min}}$-imprimitive $L$-function values in what follows.

Set $\theta = \theta_{K_1/k_0, S^{\mathrm{min}}} (0)$ and split it into two pieces,
\begin{equation*}
	\theta_{0} = L_{K_{1}/k_{0}, S^{\mathrm{min}}} \left( 0, \chi^{p} \right) e_{\chi^{p}}
\end{equation*}
and
\begin{equation*}
	\theta_{1} = \sum_{\substack{i=1\\
							i \text{ odd }\\
							i \neq p}}^{2p -1}
				L_{K_{1}/k_{0}, S^{\mathrm{min}}} \left( 0, \overline{\chi}^{i} \right) e_{\chi^{i}},
\end{equation*}
in which $e_{\chi^i}$ denotes the idempotent corresponding to $\chi^i$.  The value $L_{K_{1}/k_{0}, S^{\mathrm{min}}} \left( 0, \chi^{i} \right)$ is $0$ when $i$ is even (including when $i=0$ since $|S^{\mathrm{min}}| \geq 2$), so
$\theta = \theta_{0} + \theta_{1}.$  

\begin{proposition}\label{P:theta0divisibility}
Let $K_1/k_0$ be a degree $2p$ cyclic extension with Galois group $G$.   If some place in $S^{\mathrm{min}}$ splits in $K_0/k_0$, then $\theta_0 = 0$. Otherwise,
\begin{equation*}
	\theta_0 = \frac{1}{p} \tilde{\theta} N_H,
\end{equation*}
in which $\tilde{\theta}$ is the lift of $\theta_{K_0/k_0, S^{\mathrm{min}}}$ to $\mathbb{Z}[G]$ given by
\begin{equation*}
 \tilde{\theta} = \frac{1}{w_0} 2^{\left| S^{\mathrm{min}} \right|-2} \left| \mathfrak{C}_0
	\right| \left( 1 - \tau \right).
\end{equation*}

\end{proposition}

\begin{proof}
The inflation property of ($S^{\mathrm{min}}$-imprimitive) Artin $L$-functions shows that $L_{K_1/k_0, S^{\mathrm{min}}} \left( 0, \chi^p \right) = L_{K_0/k_0, S^{\mathrm{min}}} \left( 0, \psi \right)$, where $\psi$ is the nontrivial character on $\Gal (K_0/k_0)$.  Combining this with the definition of $\theta_0$ yields
\begin{equation*}
	\theta_0 = L_{K_0/k_0, S^{\mathrm{min}}} (0, \psi) \frac{N_H}{p} \frac{1-\tau}{2}
\end{equation*}
The formula $\theta_0 = \frac{1}{p} \tilde{\theta} N_H$ then follows from Tate's determination \cite[\S 3, (c)]{jT1980} that $L_{K_0/k_0, S^{\mathrm{min}}} (0, \psi) = 0$ when some prime of $S^{\mathrm{min}}$ splits in $K_0/k_0$ and otherwise
\begin{equation}\label{E:quadL}
	L_{K_0/k_0, S^{\mathrm{min}}} (0, \psi) = \frac{2^{\left| S^{\mathrm{min}} \right|-1} \left| \mathfrak{C}_0
	\right|}{w_0}.
\end{equation} 
\end{proof}

The preceding proposition shows that $\theta_{0}$ is in $\mathbb{Q} \left[ G \right]$; it follows that $\theta_{1}$ is in  $\mathbb{Q} \left[ G \right]$ as well. 

\begin{lemma}\label{L:theta1properties}
$\theta_1$ has the following properties:
\begin{enumerate}
	\item $N_H \theta_1 = 0$
	\item $\theta_1$ is in $\frac{1}{pq} \mathbb{Z}[G]$
\end{enumerate}
\end{lemma}

\begin{proof}
The orthogonality relations imply that $N_H e_{\chi^i} = 0$ unless $i$ is divisible by $p$.  The first condition thus results from the definition of $\theta_1$.

The integrality property of $\theta = \theta_{K_1/k_0, S^{\mathrm{min}}} (0)$ shows that $(1-\sigma) q \theta$ is in $\mathbb{Z}[G]$; hence, so is $\left( \prod_{i=1}^{p-1} (1-\sigma^i) \right) q \theta$.  Property (ii) then follows from Property (i), Proposition \ref{P:theta0divisibility}, and the computation $ \prod_{i=1}^{p-1} (1-\sigma^i) = p - N_H$ (see \cite{bS2012b} Lemma 2.6).
\end{proof}

Proposition \ref{P:theta1divisibility} gives more refined $p$-divisibility information about the coefficients of $\theta_1$.  We need a lemma:

\begin{lemma}\label{L:factor}
If $\omega$ is in $\mathbb{Q}[G]$, then there exists $\alpha_{\omega} = \sum_{i=0}^{p-2} a_{i} \sigma^{i}$ in $\mathbb{Q}[H]$ such that
\begin{equation*}
	\omega \theta_{1} = (1-\sigma)\alpha_{\omega} (1-\tau).
\end{equation*}
If $\omega \theta_{1}$ is in $\mathbb{Z}[G]$, then $\alpha_{\omega}$ can be chosen in $\mathbb{Z}[H]$.
\end{lemma}

\begin{proof}

Because all characters in the sum defining $\theta_1$ are odd, $\omega \theta_1$ has the form
\begin{equation*}
	\omega \theta_{1} = \sum_{i=0}^{p-1} c_{i} \sigma^{i} (1-\tau).
\end{equation*}
with coefficients $c_i$ in $\mathbb{Q}$.  Property (i) of Lemma \ref{L:theta1properties} shows that $\sum_{i=0}^{p-1} c_{i} = 0$.  The present lemma follows by setting
\begin{equation*}
	\alpha_{\omega} = \sum_{i=0}^{p-2} \left( \sum_{k=0}^{i} c_{k} \right) \sigma^{i}.\qedhere
\end{equation*}
\end{proof}

\begin{proposition}\label{P:theta1divisibility}
Let $K_1/k_0$ be a degree $2p$ cyclic extension with Galois group $G$. Assume that the $p$-power roots of unity in $K_1$ are not $G$-cohomologically trivial.  
\begin{enumerate}
	\item If $\mathfrak{K}_p$ is nontrivial, then $\theta_1$ has $p$-integral coefficients.
	\item If $\mathfrak{K}_p$ is trivial, then $p \theta_1$ has $p$-integral coefficients while $\theta_1$ does not.
\end{enumerate} 
\end{proposition}
 
\begin{proof}
Lemma \ref{L:factor} shows that we may write
\begin{equation*}
	\theta_1 = (1-\sigma) (1-\tau) \sum_{i=0}^{p-2} d_i \sigma^i
\end{equation*}
for some rational numbers $d_i$ with $i=0, \ldots, p-2$.

Let $\zeta_p$ be the $p$th root of unity for which $\chi(\sigma) = \zeta_p$. Proposition 2.5 in \cite{bS2012b} says
\begin{equation}\label{E:theta1norm}
	N_{\mathbb{Q}(\zeta_p)/\mathbb{Q}} \left( \chi \left( \theta_1 \right) \right) = \frac{1}{q} 2^{(p-1) \left| S_1 \right|} p^{\left| S_2 \right|} \left| \mathfrak{K} \right|
\end{equation} 
where $N_{\mathbb{Q}(\zeta_p)/\mathbb{Q}}$ is the field norm and $S_1$ and $S_2$ are subsets of $S^{\mathrm{min}}$ (in fact, since $S^{\mathrm{min}}$ is minimal, it can be shown that $S_2$ is empty).

Combining the above two equations, we find that
\begin{equation*}
	N_{\mathbb{Q}(\zeta_p)/\mathbb{Q}} \left( \sum_{i=0}^{p-2} d_i  \zeta_p^i \right) = \frac{1}{q} 2^{(p-1) (\left| S_1 \right| - 1)} p^{\left| S_2 \right| - 1} \left| \mathfrak{K} \right|
\end{equation*}

Since the $p$-power roots of unity in $K_1$ are not $G$-cohomologically trivial, $q$ is relatively prime to $p$.  Thus, if $\mathfrak{K}_p$ is nontrivial, the above norm is $p$-integral.  As the elements $\{ \, 1, \zeta_p, \ldots, \zeta_p^{p-2} \, \}$ form an integral basis for the ring of integers in $\mathbb{Q} (\zeta_p)$, the coefficients $d_i$ are all $p$-integral.  Thus, the coefficients of $\theta_1$ are $p$-integral too.

If $\mathfrak{K}_p$ is trivial, then the above norm is not $p$-integral, hence neither is $\sum_{i=0}^{p-2} d_i \zeta_p^i$.  On the other hand, the norm of $p \sum_{i=0}^{p-2} d_i \zeta_p^i$ is a $p$-integer, so as above we find that the coefficients of $p \theta_1$ are $p$-integral. 
\end{proof}

\begin{theorem}\label{T:simplecase}
Let $K_1/k_0$ be a degree $2p$ cyclic extension with Galois group $G$.   Set $w_1 = \left| \mu_{K_1} \right|$ and $p^r = \left| \mu_{K_1} \otimes \mathbb{Z}_p \right|$.  Assume that 
\begin{enumerate}
	\item $\mu_{K_1} \otimes \mathbb{Z}_p$ is not $G$-cohomologically trivial, and
	\item $\mathfrak{K}_p$ is nontrivial.
\end{enumerate}
If some place of $k_0$ splits in $K_0/k_0$ and ramifies in $k_1/k_0$, then $w_1 \theta_{K_1/k_0, S^{\mathrm{min}}} (0)$ is in $p^r \mathbb{Z}[G]$.  If no such place exists, then for $0 \leq s \leq r$, $w_1 \theta_{K_1/k_0, S^{\mathrm{min}}} (0)$ is in $p^s \mathbb{Z}[G]$ if and only if $p^{s+1}$ divides $\left| \left( \Cl_{K_0} \otimes \mathbb{Z}_p \right)^- \right|$.
\end{theorem}

\begin{proof}
Combining conditions (i) and (ii) with Property (ii) of Lemma \ref{L:theta1properties} and Proposition \ref{P:theta1divisibility}, we find that $w_1 \theta_1$ is in $p^r \mathbb{Z}[G]$.

If some place in $S^{\mathrm{min}}$ splits in $K_0/k_0$, then $\theta_0 = 0$ by Proposition \ref{P:theta0divisibility}. Such places are exactly those that split in $K_0/k_0$ and ramify in $k_1/k_0$. Hence, if such a place exists, then $w_1 \theta_{K_1/k_0, S^{\mathrm{min}}} (0)$ is in $p^r \mathbb{Z}[G]$. 

If no place in $S^{\mathrm{min}}$ splits in $K_0/k_0$, then Proposition \ref{P:theta0divisibility} shows that the coefficients of $w_1 \theta_0$ have the same $p$-valuation as $\frac{| \mathfrak{C}_0 |}{p}$. Thus, $w_1 \theta_{K_1/k_0, S^{\mathrm{min}}} (0)$ is in $p^s \mathbb{Z}[G]$ if $s \leq r$ and $p^{s+1}$ divides $\left| \mathfrak{C}_0 \right|$.  Conversely, if $\theta_{K_1/k_0, S^{\mathrm{min}}} (0)$ is in $p^s \mathbb{Z}[G]$ with $s \leq r$, then $p^{s+1}$ divides $\left| \mathfrak{C}_0 \right|$.  The result then follows from Proposition 2.1 in \cite{bS2012b}, which shows that $\mathfrak{C}_0 \otimes \mathbb{Z}_p \cong \left( \Cl_{K_0} \otimes \mathbb{Z}_p \right)^-$.
\end{proof}

\begin{theorem}\label{T:exceptionalsimplecase}
Let $K_1/k_0$ be a degree $2p$ cyclic extension with Galois group $G$.  Set $w_1 = \left| \mu_{K_1} \right|$ and $p^r = \left| \mu_{K_1} \otimes \mathbb{Z}_p \right|$.  Assume that
\begin{enumerate}
	\item $\mu_{K_1} \otimes \mathbb{Z}_p$ is not $G$-cohomologically trivial
	\item $\mathfrak{K}_p$ is trivial
	\item Some place of $k_0$ splits in $K_0/k_0$ and ramifies in $k_1/k_0$
\end{enumerate}
Then $w_1 \theta_{K_1/k_0, S^{\mathrm{min}}} (0)$ is in $p^{r-1} \mathbb{Z}[G]$, but not in $p^r \mathbb{Z}[G]$.
\end{theorem}

\begin{proof}
Combining conditions (i) and (ii) with Property (ii) of Lemma \ref{L:theta1properties} and Proposition \ref{P:theta1divisibility}, we find that $w_1 \theta_1$ is in $p^{r-1} \mathbb{Z}[G]$, but not in $p^r \mathbb{Z}[G]$.  A place of $k_0$ that splits in $K_0/k_0$ and ramifies in $k_1/k_0$ is in $S^{\mathrm{min}}$.  Proposition \ref{P:theta0divisibility} shows that $\theta_0 = 0$, which implies the theorem.
\end{proof}

Note that all Abelian degree $2p$ extensions $K_1/k_0$ for which the group of $p$-power roots of unity in $K_1$ is not cohomologically trivial fit the hypotheses of either Theorem \ref{T:simplecase} or Theorem \ref{T:exceptionalsimplecase} excepting the TK extensions.  For those, the situation requires a more careful analysis.

\section{The sophisticated case: TK extensions}\label{S:sophisticatedcase}

We retain the notation introduced at the beginning of Section \ref{S:simplecase}.  To begin, we identify properties of TK extensions that further reveal their exceptionality.  

\begin{proposition}\label{P:trivialproperties}
If $K_1/k_0$ is a TK extension, then
\begin{enumerate}
	\item $K_1/k_0$ is unramified away from $p$
	\item $\left( \Cl_{K_0} \otimes \mathbb{Z}_p \right)^-$ and $\left( \Cl_{K_1} \otimes \mathbb{Z}_p \right)^-$ are isomorphic nontrivial cyclic groups with trivial $\Gal(K_1/K_0)$-action
\end{enumerate}
\end{proposition}

\begin{proof}
Condition (i) follows from Lemma 3.4 in \cite{bS2012b} (note that the case $B^{\sharp}$ appearing in that lemma consists of extensions for which the group of $p$-power roots of unity in $K_1$ is not $G$-cohomologically trivial and for which no prime ideal of $k_0$ splits in $K_0/k_0$ and ramifies in $k_1/k_0$).

The isomorphism in condition (ii) follows immediately from the surjectivity of the norm map $N \colon \mathfrak{C}_{1,p} \rightarrow \mathfrak{C}_{0,p}$ and the triviality of its kernel.  Because $N$ is a $H$-isomorphism, $H$ acts trivially on $\mathfrak{C}_{1,p}$.  The cyclicity follows from Lemma 3.3 in \cite{bS2012b} and the assumption that $\mathfrak{K}_p$ is trivial.  Finally, the nontriviality in (ii) is a consequence of Lemma 3.5 in \cite{bS2012b}.
\end{proof}

In the next proposition, we set for convenience $\varepsilon (\mathfrak{a}) = \varepsilon_{\N_{K_{1}/K_{0}} \mathfrak{a}, K_{0}/k_{0}, S^{\mathrm{min}}}$, an element (defined only up to a factor of a root of unity) associated with $N_{K_1/K_0} \mathfrak{a}$ by the Brumer-Stark conjecture for $K_0/k_0$ and $S^{\mathrm{min}}$.

\begin{proposition}\label{P:quadcondition}
Assume that no prime ideal of $k_0$ splits in $K_0/k_0$ and ramifies in $k_1/k_0$ and that the $p$-power roots of unity in $K_1$ are not $G$-cohomologically trivial. For each ideal $\mathfrak{a}$ representing a class in $\Cl_{K_1} \otimes \mathbb{Z}_p$, the element $\varepsilon(\mathfrak{a})$ is in $K_1^{\times p} \mu_{K_1}$.
\end{proposition}

\begin{proof}
Set $\theta' = \theta_{K_0/k_0, S^{\mathrm{min}}} (0)$. 

First, suppose that the $p$-rank of the group $\mathfrak{C}_{0}$ is at least $2$.  Equation \eqref{E:quadL} implies that (see \cite{bS2012b} Lemma 3.1 for details) 
\begin{equation*}
	\N_{K_{1}/K_{0}} \mathfrak{a}^{\frac{w_{0}}{p} \theta'} = \left( \gamma \right)
\end{equation*}
for some anti-unit $\gamma$ in $K_0$.  It follows that $\left( \varepsilon (\mathfrak{a}) \right) = \left( \gamma^{p} \right)$, and the theorem then follows from the fact that both of $\varepsilon (\mathfrak{a})$ and $\gamma^{p}$ are anti-units.

Otherwise, suppose that $\mathfrak{C}_{0,p} \cong \mathbb{Z}/p^{t} \mathbb{Z}$. Using Tate's determination of the equivariant $L$-function value at $0$ for quadratic extensions (\cite[\S 3, (c)]{jT1980}), the lift of the ideal $(\varepsilon (\mathfrak{a}))$ to $K_{1}$ is then equal to
\begin{equation*}
	\mathfrak{a}^{N_{H} 2^{|S^{\mathrm{min}}|-2} \left|
	\mathfrak{C}_{0} \right| (1-\tau)} = \mathfrak{a}^{N_H p^{t} c (1-\tau)},
\end{equation*}
for some integer $c$.  As $\mathfrak{a}^{N_{H}}$ is the lift of a fractional ideal of $K_{0}$, Lemma 3.5 in \cite{bS2012b} shows that the class of $\mathfrak{a}^{N_{H}}$ has order dividing $p^{t-1}$ in $A_{1}$.  Thus, there is an element $\tilde{\gamma}$ in $K_{1}^{\times}$ such that
\begin{equation*}
	\mathfrak{a}^{N_{H} p^{t-1} (1-\tau)} = \left( \tilde{\gamma}^{1-\tau} \right),
\end{equation*}
and thus, $\varepsilon (\mathfrak{a})$ and $\tilde{\gamma}^{p c (1-\tau)}$ generate the same principal ideal in $K_1$.
As $\tilde{\gamma}^{p c (1-\tau)}$ and $\varepsilon (\mathfrak{a})$ are anti-units, they differ by a factor of a root of unity.
\end{proof}

In the next proposition, we write $\mu_p$ for the group of $p$th roots of unity in $K_0$.

\begin{proposition}\label{P:torsionunit}
Assume that no prime ideal of $k_0$ splits in $K_0/k_0$ and ramifies in $k_1/k_0$ and that that the $p$-power roots of unity in $K_1$ are not $G$-cohomologically trivial. Set $p^r = \left| \mu_{K_1} \otimes \mathbb{Z}_p \right|$.  For each $\sigma$ in $H=\Gal(K_1/K_0)$ and ideal $\mathfrak{a}$ in $K_{1}$ representing a class in $\Cl_{K_{1}} \otimes \mathbb{Z}_p$, there exists an anti-unit $\gamma_{\sigma} \left( \mathfrak{a} \right)$ in  $K_{1}$ for which
\begin{enumerate}
	\item $\mathfrak{a}^{(\sigma-1) q \theta_1} = \left( \gamma_{\sigma} \left( \mathfrak{a} \right) \right)$.
	\item $\N_{K_{1}/K_{0}} \left( \gamma_{\sigma} \left( \mathfrak{a} \right) \right)$ is a $p^r$th root of 
	unity.
\end{enumerate}
Moreover, the map from $H$ to $\mu_p$ given by 
\begin{equation*}
	 \sigma \rightarrow \N_{K_{1}/K_{0}} \left( \gamma_{\sigma} \left( \mathfrak{a} \right) \right)^{p^{r-1}}
\end{equation*}
is a well defined cohomology class in $H^1(H, \mu_p) = \Hom (H, \mu_p)$. 
\end{proposition}

\begin{proof}  
Set $\theta = \theta_{K_1/k_0, S^{\mathrm{min}}} (0)$.  Proposition \ref{P:theta0divisibility} shows that $(\sigma-1) q \theta = (\sigma-1) q \theta_{1}$. The Integrality Property of $\theta$ implies that $(\sigma-1) q \theta$ is in $\mathbb{Z}[G]$, hence so is $(\sigma-1) q \theta_1$. Lemma \ref{L:factor} then shows that we may factor
\begin{equation*}
	(\sigma-1) q \theta_1 = (\sigma-1)\alpha(1-\tau),
\end{equation*}
with $\alpha$ in $\mathbb{Z}[H]$.  
  
Proposition 2.5 in \cite{bS2012b} gives
\begin{equation*}
	N_{\mathbb{Q} \left( \zeta_{p} \right)/\mathbb{Q}} \left( \chi \left( \alpha \right) \right) = 
	q^{p-2} 2^{(p-1) \left( \left| S_{1} \right| - 1 \right)} p^{\left| S_2 \right|}  \left| \mathfrak{K} \right| 
\end{equation*}
for some subsets $S_1$ and $S_2$ of $S^{\mathrm{min}}$.
Let $\mathcal{O} = \mathbb{Z} \left[ \zeta_{p} \right]$ be the ring of integers in the cyclotomic field $\mathbb{Q} \left( \zeta_{p} \right)$.  Through the isomorphism
\begin{equation*}
	\chi: \mathbb{Z} [H]/N_{H} \rightarrow \mathcal{O},
\end{equation*}
$\mathfrak{K}$ is endowed with an $\mathcal{O}$-module structure.  Hence, we may rewrite the above expression for the norm of $\chi ( \alpha)$ in terms of absolute norms of ideals in $\mathcal{O}$:
\begin{equation*}
	\mathfrak{N} \left( \left( \chi \left( \alpha \right) \right) \right)=
	q^{p-2} \mathfrak{N} \left( 2^{\left| S_{1} \right|-1} (1-\zeta_p)^{\left| S_2 \right|} 
	\Fit_{\mathcal{O}} \left( \mathfrak{K} \right) \right),
\end{equation*}
in which $\Fit_{\mathcal{O}}$ denotes the zeroth Fitting ideal.  Since there is only one prime ideal in $\mathcal{O}$ dividing $p$, we find that $\alpha$ annihilates $\mathfrak{K}_{p}$.

Now $\mathfrak{a}^{\sigma-1}$ represents a class in $\mathfrak{K}_{p}$.  Thus, $\mathfrak{a}^{(\sigma-1) \alpha } = \mathfrak{b} \mathfrak{c} (\eta)$, where $\mathfrak{b}$ is an ideal in $K_{1}$ divisible only by primes in $S_{K_{1}}^{\ns}$, $\mathfrak{c}$ is the lift of an ideal from $k_{1}$ to $K_{1}$, and $\eta$ is in $K_{1}^{\times}$.  Since $\mathfrak{b}$ and $\mathfrak{c}$ are both fixed by $\tau$, it follows that
\begin{equation*}
	 \mathfrak{a}^{(\sigma-1) q \theta_1} = 
	 \mathfrak{a}^{(\sigma-1) \alpha (1-\tau)}
	= \left( \eta^{1-\tau} \right).
\end{equation*}

We see immediately that $N_{K_1/K_0} (\eta^{1-\tau})$ is an anti-unit in $K_0$ and that
\begin{equation*}
	\left(   \N_{K_{1}/K_{0}} \left( \eta^{1-\tau} \right) \right)
	= \mathfrak{a}^{(\sigma-1) q \theta N_{H}} = (1).
\end{equation*}
It follows that $N_{K_{1}/K_{0}} \left( \eta^{1-\tau} \right)$ is a root of unity in $K_0$. Since the $p$-power roots of unity in $K_1$ are not cohomologically trivial, they are all contained in $K_0$. Write $\N_{K_{1}/K_{0}} \left( \eta^{1-\tau} \right) = \zeta_{p^{r}}^{j} \xi^{-p}$,
where $\xi$ is a root of unity in $K_0$ of order relatively prime to $p$.   Then the element 
\begin{equation*}
	\gamma_{\sigma} (\mathfrak{a}) = \xi \eta^{1-\tau}
\end{equation*}
satisfies  conditions (i) and (ii) of the proposition.

Because the element $\gamma_{\sigma} (\mathfrak{a})$ is an anti-unit, conditions (i) and (ii) specify it uniquely up to multiplication by a root of unity $\zeta'$ in $\mu_{1}$ such that $\N_{K_{1}/K_{0}} (\zeta')$ is a $p$-power root of unity.  If $\zeta' = \zeta_{p^{r}}^{k} \xi'$ where $\xi'$ has order relatively prime to $p$, then $N_{K_{1}/K_{0}} \left( \xi' \right) = 1$.  Therefore, since $\zeta_{p^r}$ is in $K_0$, $N_{K_{1}/K_{0}} (\zeta')^{p^{r-1}} = 1$.  It follows that $N_{K_{1}/K_{0}} \left( \gamma_{\sigma} \left(\mathfrak{a} \right) \right)^{p^{r-1}}$ is independent of the choice of $\gamma_{\sigma} (\mathfrak{a})$.

Finally, the factorization $\sigma^n - 1 = (\sigma-1)(\sum_{i=0}^{n-1} \sigma^i)$ shows that for a fixed generator $\sigma$ of $H$, we may choose
\begin{equation*}
	\gamma_{\sigma^n}(\mathfrak{a}) = \gamma_{\sigma}(\mathfrak{a})^{\sum_{i=0}^{n-1} \sigma^i}
\end{equation*}
With this choice,
\begin{equation*}
	N_{K_1/K_0} (\gamma_{\sigma^n}(\mathfrak{a})) = N_{K_1/K_0} (\gamma_{\sigma}(\mathfrak{a}))^n,
\end{equation*}
so the map in the proposition is indeed a well-defined element of $H^1 (H, \mu_p)$.
\end{proof}

\begin{theorem}\label{T:condition1}
Set $\theta = \theta_{K_1/k_0, S^{\mathrm{min}}} (0)$. Assume that no prime ideal of $k_0$ splits in $K_0/k_0$ and ramifies in $k_1/k_0$ and that the $p$-power roots of unity in $K_1$ are not $G$-cohomologically trivial. For every fractional ideal $\mathfrak{a}$ representing a class in  $\Cl_{K_{1}} \otimes \mathbb{Z}_p$, $\mathfrak{a}^{w_1 \theta}$ is principal, generated by an anti-unit.
\end{theorem}

\begin{proof}
Since $p$ divides $w_0$, Lemma \ref{L:theta1properties} shows that $w_1 \theta_1$ is in $\mathbb{Z}[G]$; hence, the Integrality Property of $\theta$ shows $w_1 \theta_0$ is in $\mathbb{Z}[G]$.  We prove the theorem by showing each of $\mathfrak{a}^{w_1 \theta_0}$ and $\mathfrak{a}^{w_1 \theta_1}$ is principal, generated by an anti-unit.

Choose $\gamma_{\sigma}(\mathfrak{a})$ satisfying $\mathfrak{a}^{(\sigma-1)q \theta_1} = \left( \gamma_{\sigma}(\mathfrak{a}) \right)$ as in Proposition \ref{P:torsionunit}. We compute $\prod_{i=1}^{p-1} ( \sigma^i-1) = p - N_H$ (see \cite{bS2012b} Lemma 2.6) and use Lemma \ref{L:theta1properties} to find
\begin{equation*}
	\mathfrak{a}^{w_1 \theta_1} = \left( \gamma_{\sigma} (\mathfrak{a})^{\frac{w_0}{p} \prod_{i=2}^{p-1}
	(\sigma^i-1)} \right).
\end{equation*}

Next, letting $\mathcal{O}_{K_1}$ denote the ring of integers in $K_1$, Proposition \ref{P:theta0divisibility} shows that
\begin{align*}
	\mathfrak{a}^{pw_1 \theta_0} 
	&= \mathfrak{a}^{w_1 \tilde{\theta} N_H}\\
	&= \left( N_{K_1/K_0} \mathfrak{a} \right)^{q w_0 \theta_{K_0/k_0, S^{\mathrm{min}}}} \mathcal{O}_{K_1}
\end{align*}
Since $K_0$ and $K_1$ have the same number of $p$-power roots of unity,  Proposition \ref{P:quadcondition} shows that we can choose $\varepsilon(\mathfrak{a})$ (defined only up to a factor of a root of unity) so that $\sqrt[p]{\varepsilon(\mathfrak{a})}$ is an anti-unit in $K_1$. It follows that
\begin{equation*}
	\mathfrak{a}^{p w_1 \theta_0} = \left( \sqrt[p]{\varepsilon(\mathfrak{a})}^{\, q} \right)^p.
\end{equation*}
Unique factorization then proves $\mathfrak{a}^{w_1 \theta_0}$ is principal, generated by the anti-unit $\sqrt[p]{\varepsilon(\mathfrak{a})}^{\, q}$.

Finally,
\begin{equation}\label{E:BSgenerator}
	\mathfrak{a}^{w_1 \theta} =  \left( \gamma_{\sigma} (\mathfrak{a})^{\frac{w_0}{p} \prod_{i=2}^{p-2}
	(\sigma^i-1)} \cdot \sqrt[p]{\varepsilon(\mathfrak{a})}^{\, q} \right)\qedhere
\end{equation}
\end{proof}

Our next result shows when the $p$-power roots of unity in $K_1$ have nontrivial cohomology and no prime ideal of $k_0$ splits in $K_0/k_0$ and ramifies in $k_1/k_0$, the $p$-local Brumer-Stark conjecture for $K_{1}/k_{0}$ is equivalent to the equality of two arithmetically defined cohomology classes in $H^1 (H, \mu_p)$.   For each fractional ideal $\mathfrak{a}$ of $K_{1}$ representing a class in $\Cl_{K_{1}} \otimes \mathbb{Z}_p$, we let $\varepsilon (\mathfrak{a})$ be an element in $K_0 \cap K_1^{\times p}$ associated with $N_{K_1/K_0} \mathfrak{a}$ by the Brumer-Stark conjecture for $K_0/k_0$. For each $\sigma$ in $H = \Gal(K_1/K_0)$, we let $\N_{K_{1}/K_{0}} \left( \gamma_{\sigma} (\mathfrak{a}) \right)^{p^{r-1}}$ be the $p$th root of unity from Proposition \ref{P:torsionunit}. For each $\alpha$ in $K_{0} \cap K_1^{\times p}$, we let $\left( \sigma, \alpha \right)_{\Kum}$ denote the Kummer pairing with the element $\sigma$ of $H$.

\begin{theorem}\label{T:recharacterization}
Set $\theta = \theta_{K_1/k_0, S^{\mathrm{min}}} (0)$. Assume that no prime ideal of $k_0$ splits in $K_0/k_0$ and ramifies in $k_1/k_0$ and that the $p$-power roots of unity in $K_1$ are not $G$-cohomologically trivial.    The existence of an element $\varepsilon_{\mathfrak{a}, K_1/k_0, S^{\mathrm{min}}}$ satisfying the conditions of the $p$-local Brumer Stark conjecture for $K_1/k_0$:
	  \renewcommand{\theenumii}{\roman{enumii}}
		\begin{enumerate}
				\item $\mathfrak{a}^{w_1 \theta} = \left( \varepsilon_{\mathfrak{a}, K_1/k_0, S^{\mathrm{min}}} \right)$
	\item $\varepsilon_{\mathfrak{a}, K_1/k_0, S^{\mathrm{min}}}$ is an anti-unit
	\item $\varepsilon_{\mathfrak{a}, K_1/k_0, S^{\mathrm{min}}}$ is $p^r$-Abelian.
	  \end{enumerate}
is equivalent to the equality
		\begin{equation*}
			\left( \sigma, \varepsilon (\mathfrak{a}) \right)_{\Kum}^{q}
			= \N_{K_{1}/K_{0}} \left( \gamma_{\sigma} (\mathfrak{a}) \right)^{w_0/p},
		\end{equation*}
\end{theorem}

\begin{remark}
That $\varepsilon(\mathfrak{a})$ can be chosen in $K_0 \cap K_1^{\times p}$ is a consequence of $K_0$ and $K_1$ having the same number of $p$-power roots of unity and Proposition \ref{P:quadcondition}. Note that our choice of $\varepsilon(\mathfrak{a})$ only defines it up to a factor of a root of unity in $K_0$ of order dividing $\frac{w_0}{p}$, but the condition on the Kummer pairing is unaffected by such a factor. 
\end{remark}

\begin{proof}
Theorem \ref{T:condition1} shows that properties (i) and (ii) hold without contingency: equation \eqref{E:BSgenerator} shows $\mathfrak{a}^{w_1 \theta} = (\eta)$, where
\begin{equation*}
	\eta = \gamma_{\sigma} (\mathfrak{a})^{\frac{w_0}{p} \prod_{i=2}^{p-1}
	(\sigma^i-1)}  \sqrt[p]{\varepsilon(\mathfrak{a})}^{\, q}.
\end{equation*}
Let $N_{\sigma}$ be an integer for which $\zeta^{\sigma} = \zeta^{N \sigma}$ for all roots of unity $\zeta$ in $\mu_{K_1}$. Saying $K_1 \left( \sqrt[\uproot{2} p^r]{\eta} \right)/k_0$ is Abelian is equivalent to saying $\eta^{\sigma-N \sigma}$ is in $K_{1}^{\times p^{r}}$ 
(see, for instance, the end of the proof of case $A^{\sharp}$ of Theorem 4.2 in \cite{bS2012b}.)  Since $N \sigma \equiv 1 \pmod{p^{r}}$, this is equivalent to
\begin{equation*}
	\eta^{\sigma-1} = \gamma_{\sigma} (\mathfrak{a})^{w_{0} - \frac{w_0}{p} N_{H}} 
	\sqrt[p]{\varepsilon}^{\, q(\sigma-1)} \in K_{1}^{\times p^{r}}.
\end{equation*}
Finally, this is equivalent to
\begin{equation*}
	\left( \sigma, \varepsilon \right)_{\Kum}^q \left( \N_{K_{1}/K_{0}} 
	\left( \gamma_{\sigma} (\mathfrak{a}) \right) \right)^{-\frac{w_0}{p}} \in K_{1}^{\times p^{r}}.
\end{equation*}
Since the quantity on the left is a $p$th root of unity, it is $1$, the condition of the theorem.  
\end{proof}

\begin{remark}
If the $p$-rank of $\mathfrak{C}_0$ is at least $2$, then $(\sigma, \varepsilon(\mathfrak{a}))_{\Kum} = 1$ follows immediately using Proposition \ref{P:theta0divisibility}.  Using Lemma 3.3 in \cite{bS2012b} and Equation \eqref{E:theta1norm}, it is possible to show that $N_{K_1/K_0} (\gamma_{\sigma} (\mathfrak{a}))^{w_0/p} = 1$.  Theorem \ref{T:recharacterization} then shows that the $p$-primary Brumer-Stark conjecture holds under the hypotheses of the theorem and the additional hypothesis $\rk_p(\mathfrak{C}_0) \geq 2$.  This is also a consequence of Theorem 4.2 in \cite{bS2012b}.
\end{remark}

By the above remark, the cohomological condition of Theorem \ref{T:recharacterization} is most interesting when the $p$-rank of $\mathfrak{C}_{0}$ is $1$ (it is impossible under the conditions of \ref{T:recharacterization} for $\mathfrak{C}_0 \otimes \mathbb{Z}_p$ to be trivial --- see \cite{bS2012b} Lemma 3.5). Proposition \ref{P:trivialproperties} shows that this happens in particular for TK extensions.  For those, we will provide an alternative characterization of the $p$-primary Brumer-Stark conjecture.  This will make explicit a connection between the Brumer-Stark conjecture and the coefficients of $\theta_{K_1/k_0, S^{\mathrm{min}}} (0)$.    

\begin{definition}
Let $K_1/k_0$ be a TK extension, and set $p^r = \left| \mu_{K_1} \otimes \mathbb{Z}_p \right|$. For $1 \leq i \leq r$, let $\mu_{p^i}$ be the group of $p^i$th roots of unity in $K_1$. For each integer $t \geq 1$, the \textbf{$n$th modified radical group} associated to the extension $K_1/K_0$ is
\begin{equation*}
	\mathfrak{W}_n = \{ \, \omega \in K_1 \mid \omega^{\sigma-1} \in K_1^{\times p^n} \mu_{p^r} \, \}.
\end{equation*}
\end{definition}

If $1 \leq n_1 \leq n_2$, then $\mathfrak{W}_{n_1} \supseteq \mathfrak{W}_{n_2}$.  Also, $\cap_{n \geq 1} \mathfrak{W}_{n}$ is the usual radical group for the extension, consisting of those $\omega$ for which $\omega^{\sigma-1}$ is in $\mu_p$.  

If $\omega \in \mathfrak{W}_n$ and $\omega^{\sigma-1} = \delta_{\sigma}^{p^n} \zeta$, then we write
\begin{equation*}
	\left( \sigma, \omega \right)_{\Kum} = \zeta \quad \text{ and } \quad\left( \sigma, \omega \right)_{\N} = N_{K_1/K_0} \delta_{\sigma}.
\end{equation*}
Note that there is a unique decomposition $\omega^{\sigma - 1} = \delta_{\sigma}^{p^n} \zeta$ when $n \geq r$, but in general, the symbols need not be well defined. The symbols $\left( \cdot, \cdot \right)_{\Kum}$ and $\left( \cdot, \cdot \right)_{\N}$ have the following properties:

\begin{symbolproperties}\hfill
\begin{enumerate}
	\item Both $\left( \sigma, \omega \right)_{\N}^{p^{r-1}}$ and the product $\left( \sigma, \omega 
		\right)_{\Kum} \left( \sigma, \omega \right)_{\N}^{p^{n-1}}$ are elements of $\mu_p$ and are independent of the decomposition $\omega = \delta_{\sigma}^{p^n} \zeta$.\\
	\item If $n \geq r$, then $\left( \sigma, \omega \right)_{\Kum}$ is in $\mu_p$.\\
	\item The maps $\left( \cdot, \cdot \right)_{\N}^{p^{r-1}}$, $\left( \cdot, \cdot \right)_{\Kum} (\cdot, \cdot)_{\N}^{p^{n-1}}$, and (when $n \geq r$)
		$(\cdot, \cdot)_{\Kum}$ are bihomomorphisms, hence define elements of $H^1 (H, \mu_p)$\\
	\item The maps  $\left( \sigma, \cdot \right)_{\Kum} (\sigma, \cdot)_{\N}^{p^{n-1}}$ and 
		$(\sigma, \cdot)_{\Kum}$ in cases (i) and (ii) respectively
		restrict to the Kummer pairing $(\sigma, \cdot )_{\Kum}$ on the usual radical group.\\
	\item If $n \geq r$, the kernel of $\left( \sigma, \cdot \right)_{\Kum}$
		consists of the $p^r$-Abelian elements for $K_1/K_0$ in $\mathfrak{W}_n$.
\end{enumerate}
\end{symbolproperties}

\begin{lemma}\label{L:modifiedradical}
Let $K_1/k_0$ be a TK extension.  Let $\mathfrak{a}$ be an ideal representing a class in $\Cl_{K_1} \otimes \mathbb{Z}_p$.  Let $p^t = \left| \mathfrak{C}_{0,p} \right|$. Then $\mathfrak{a}^{p^t (1-\tau)} = \left( \omega(\mathfrak{a}) \right)$ for some anti-unit $\omega(\mathfrak{a})$ in the modified radical group $\mathfrak{W}_t$ of $K_1/K_0$.  
\end{lemma}

\begin{proof}
Proposition \ref{P:trivialproperties} and Lemma 2.1 in \cite{bS2012b} show that $\mathfrak{C}_{0,p}$ and $\mathfrak{C}_{1,p}$ are both isomorphic cyclic  groups with trivial $H$-action.  As $\mathfrak{C}_{1}$ is cyclic of order $p^t$, there exists an anti-unit $\omega(\mathfrak{a})$ in $K_1$ satisfying $\mathfrak{a}^{p^{t}(1-\tau)} = \left( \omega(\mathfrak{a}) \right)$. As $\mathfrak{K}_p$ is trivial, there exists an anti-unit $\delta_{\sigma}(\mathfrak{a})$ in $K_{1}$ such that $\mathfrak{a}^{(\sigma-1)(1-\tau)} = \left( \delta_{\sigma}(\mathfrak{a})\right)$.    The anti-units $\omega(\mathfrak{a})^{\sigma-1} $ and $\delta_{\sigma}(\mathfrak{a})^{p^{t}}$ thus generate the same ideal of $K_1$.  Hence, there exists a root of unity $\zeta$ in $K_1$ such that
\begin{equation}\label{E:differbyzeta}
	\omega(\mathfrak{a})^{\sigma-1} =  \zeta \delta_{\sigma}(\mathfrak{a})^{p^{t}}.
\end{equation}
Adjusting $\zeta$ and $\delta_{\sigma}(\mathfrak{a})$ by roots of unity of order relatively prime to $p$ if necessary, we can arrange for $\zeta$ to be a $p$-power root of unity.  The above equation then shows that $\omega(\mathfrak{a})$ is in $\mathfrak{W}_t$.  
\end{proof}

\begin{theorem}\label{T:bsharprank1}
Let $K_{1}/k_{0}$ be a TK extension with Galois group $G$.  Let $b_0$ and $b_1$ be the coefficients of the identity automorphism $1_G$ in $\theta_0$ and $\theta_1$, so $p b_0 = \zeta_{K_0/k_0, S^{\mathrm{min}}} (1_G, 0)$ and $b_0 + b_1 = \zeta_{K_1/k_0, S^{\mathrm{min}}} (1_G, 0)$.  Let $p^r = \left| \mu_{K_1} \otimes \mathbb{Z}_p \right|$ and $p^t = \left| \mathfrak{C}_{0,p} \right|$.  Fix $\mathfrak{a}$ representing a class in $\Cl_{K_1} \otimes \mathbb{Z}_p$, and let $\omega(\mathfrak{a})$ be the element of $\mathfrak{W}_t$ from Lemma \ref{L:modifiedradical}. The $p$-local Brumer-Stark conjecture for $K_1/k_0$ holds if and only if
\begin{alignat*}{2}
		\left[ \left( \sigma, \omega(\mathfrak{a}) \right)_{\Kum} \left( \sigma, \omega(\mathfrak{a})
	 	\right)_{\N}^{p^{t-1}} \right] \left[ \left( \sigma, \omega(\mathfrak{a}) \right)_{\N}^{p^{r-1}} \right]^{
		p^{t-r} \frac{b_1}{b_0}} &= 1, \quad &\text{if $t  < r$;}\\
		\left( \sigma, \omega(\mathfrak{a}) \right)_{\Kum} \left[ \left( \sigma, \omega(\mathfrak{a}) 
		\right)_{\N}^{p^{r-1}} \right]^{\frac{b_0+b_1}{b_0}}
		&=1, \quad &\text{if $t = r$.}\\
		\left( \sigma, \omega(\mathfrak{a}) \right)_{\Kum} \left[ \left( \sigma, \omega(\mathfrak{a}) 
		\right)_{\N}^{p^{r-1}} \right]^{p^{t-r} \frac{b_1}{b_0}}
		&=1, \quad &\text{if $t > r$.}
\end{alignat*}
\end{theorem}

\begin{remark}
Equation \eqref{E:quadL} and the following proof will show that $p^{r-t} \frac{b_1}{b_0}$ is in $\mathbb{Z}_p^{\times}$. The expressions were written in this complicated fashion to make explicit the appearance of the bihomomorphisms from above.  Each equation in the theorem expresses the triviality of a certain class in $H^1 (H, \mu_p)$. 
\end{remark}

\begin{proof}
First, observe using Proposition \ref{P:theta0divisibility} that $\frac{w_0}{p^{t-1}} b_0$ is an integer  prime to $p$.  Also, compute $(\sigma-1)(\sum_{i=0}^{p-1} (p-1-i) \sigma^i) = N_H-p$.   Define $\zeta$ and $\delta_{\sigma} (\mathfrak{a})$ through the relation $\omega(\mathfrak{a})^{\sigma-1} = \zeta \delta_{\sigma} (\mathfrak{a})^{p^t}$ (see Equation \eqref{E:differbyzeta}). The ideal in $K_1$ generated by the element $\varepsilon (\mathfrak{a})$ associated with the ideal $N_{K_1/K_0} \mathfrak{a}$ by the Brumer-Stark conjecture for $K_0/k_0$ and $S^{\mathrm{min}}$ is
\begin{equation*}
	\left( \varepsilon (\mathfrak{a}) \right) = \mathfrak{a}^{N_H p w_0 b_0 (1-\tau)} =
	\left( \left( \omega(\mathfrak{a})^p \delta_{\sigma}(\mathfrak{a})^{p^{t}\sum_{j=0}^{p-1} (p-1-j) \sigma^{j}} 
	 \right)^{\frac{w_0}{p^{t-1}} b_0} \right).
\end{equation*}
The choice of $\varepsilon (\mathfrak{a})$ in $K_0 \cap K_1^{\times p}$ per the remark following the statement of Theorem \ref{T:recharacterization} can then be the generator on the right.  Then
\begin{align*}
	\left( \sigma, \varepsilon (\mathfrak{a}) \right)_{\Kum} &= \left[ \omega(\mathfrak{a})^{\sigma-1} \cdot
	\delta_{\sigma}(\mathfrak{a})^{p^{t-1}(\sigma-1)\sum_{j=0}^{p-1} (p-1-j) \sigma^{j}} 
	\right]^{\frac{w_0}{p^{t-1}} b_0}\\
	&= \left[ \zeta \delta_{\sigma}(\mathfrak{a})^{p^t} \cdot \delta_{\sigma}(\mathfrak{a})^{p^{t-1} N_{H}-p^t}
	 \right]^{\frac{w_0}{p^{t-1}} b_0}\\
	 &= \left[ \zeta \N_{K_{1}/K_{0}} \left( \delta_{\sigma}(\mathfrak{a}) \right)^{p^{t-1}}
	 \right]^{\frac{w_0}{p^{t-1}} b_0}\\
	 &= \left[ \left( \sigma, \omega(\mathfrak{a}) \right)_{\Kum} \left( \sigma, \omega(\mathfrak{a})
	 \right)_{\N}^{p^{t-1}} \right]^{\frac{w_0}{p^{t-1}} b_0}
\end{align*}

We will now express the $p$th root of unity $N_{K_1/K_0} \left( \gamma_{\sigma} (\mathfrak{a}) \right)^{p^{r-1}}$ from Proposition \ref{P:torsionunit} in terms of $\delta_{\sigma} (\mathfrak{a})$.  Write 
 \begin{equation*}
 	(\sigma-1) q \theta_1 = (\sigma-1) \alpha (1-\tau)
\end{equation*}
with $\alpha$ in $\mathbb{Z}[H]$ as in the proof of  Proposition \ref{P:torsionunit}.  
By definition, an element $\gamma_{\sigma} (\mathfrak{a})$ as given by Proposition \ref{P:torsionunit} generates the ideal
\begin{equation*}
	\left( \gamma_{\sigma} (\mathfrak{a}) \right) = \mathfrak{a}^{(\sigma-1)(1-\tau) \alpha}
	= \left( \delta_{\sigma}(\mathfrak{a})^{\alpha} \right).
\end{equation*}
Equation \eqref{E:differbyzeta} shows that $ N_{K_1/K_0} (\delta_{\sigma} (\mathfrak{a}))$ is a $p$-power root of unity.  Thus, we may choose $\gamma_{\sigma} (\mathfrak{a})$ as per Proposition \ref{P:torsionunit} to be $\delta_{\sigma}(\mathfrak{a})^{\alpha}$.  Then 
\begin{equation*}
	N_{K_1/K_0} \left( \gamma_{\sigma} (\mathfrak{a}) \right)^{p^{r-1}} = \N_{K_{1}/K_{0}} \left( \delta_{\sigma}(\mathfrak{a}) \right)^{p^{r-1} 
	\alpha}.
\end{equation*}

The condition $\left( \sigma, \varepsilon \right)_{\Kum}^q = N_{K_1/K_0} \left( \gamma_{\sigma} (\mathfrak{a}) \right)^{w_0/p}$ of Theorem \ref{T:recharacterization} is thus equivalent to
\begin{equation}\label{E:characterrelation}
	\left[ \left( \sigma, \omega(\mathfrak{a}) \right)_{\Kum} \left( \sigma, \omega(\mathfrak{a})
	 \right)_{\N}^{p^{t-1}} \right]^{\frac{w_1}{p^{t-1}} b_0}
	= \N_{K_{1}/K_{0}} \left( \delta_{\sigma}(\mathfrak{a}) \right)^{  \frac{w_0}{p} \alpha}
\end{equation}

Lemma \ref{L:theta1properties} shows that $pq b_1$ is an integer, and Proposition \ref{P:theta1divisibility} shows that it is relatively prime to $p$ (note that $q$ is relatively prime to $p$ since $\mu_{K_1} \otimes \mathbb{Z}_p$ is not cohomologically trivial). 

The definition of $\alpha$ shows that $q \theta_1$ and $\alpha (1-\tau)$ differ by a multiple of $N_H (1-\tau)$, say 
\begin{equation*}
	\alpha (1-\tau) = q \theta_1 + c N_H (1-\tau).
\end{equation*}  As $\alpha$ is in $\mathbb{Z}[H]$, $cp \equiv -pqb_1 \pmod{p}$.  Applying $N_H$ to both sides of the above equation then yields, using Lemma \ref{L:theta1properties}
\begin{equation*}
	\alpha N_H (1-\tau) = c p N_H (1-\tau) \equiv -pq b_1 N_H (1-\tau) \pmod{p \mathbb{Z}[G]},
\end{equation*} 
hence $\alpha N_H \equiv -p q b_1 N_H \pmod{p N_H \mathbb{Z}[H]}$.

Since $N_{K_1/K_0} \left( \delta_{\sigma} (\mathfrak{a}) \right)^{\frac{w_0}{p}}$ is a $p$th root of unity, we may rewrite the right side of Equation \eqref{E:characterrelation}:
\begin{align*}
	\left[ \left( \sigma, \omega(\mathfrak{a}) \right)_{\Kum} \left( \sigma, \omega(\mathfrak{a})
	 \right)_{\N}^{p^{t-1}} \right]^{\frac{w_1}{p^{t-1}} b_0}
	&= \delta_{\sigma} (\mathfrak{a})^{\frac{w_0}{p} \alpha N_H}\\
	&= N_{K_1/K_0} \left( \delta_{\sigma} (\mathfrak(a)) \right)^{-w_1 b_1}
\end{align*} 
  By definition, $\left( \sigma, \omega(\mathfrak{a})
	 \right)_{\N} = N_{K_1/K_0} \left( \delta_{\sigma} (\mathfrak{a}) \right)$.  Thus,
\begin{equation*}
		\left[ \left( \sigma, \omega(\mathfrak{a}) \right)_{\Kum} \left( \sigma, \omega(\mathfrak{a})
	 \right)_{\N}^{p^{t-1}} \right]^{\frac{w_1}{p^{t-1}} b_0}  \left[ \left( \sigma, \omega(\mathfrak{a})
	 \right)_{\N}^{p^{r-1}} \right]^{\frac{w_1}{p^{r-1}} b_1} = 1
\end{equation*}
Observe using Proposition \ref{P:theta0divisibility} and Lemma \ref{L:theta1properties} that $\frac{w_1}{p^{t-1}} b_0$ and $\frac{w_1}{p^{r-1}} b_1$ are in $\mathbb{Z}_p^{\times}$.  Taking $\frac{w_1}{p^{t-1}} b_0$-th roots of both sides gives
\begin{equation*}
		\left[ \left( \sigma, \omega(\mathfrak{a}) \right)_{\Kum} \left( \sigma, \omega(\mathfrak{a})
	 \right)_{\N}^{p^{t-1}} \right] \left[ \left( \sigma, \omega(\mathfrak{a})
	 \right)_{\N}^{p^{r-1}} \right]^{p^{t-r} \frac{b_1}{b_0}} = 1
\end{equation*}
The theorem now follows since $(\sigma, \omega(\mathfrak{a}))_N^{p^{t-1}} = 1$ if $t > r$.
\end{proof}

We now prove our main result:  
\begin{theorem1point3}
Let $K_1/k_0$ be a TK extension for which the $p$-local Brumer-Stark Conjecture holds.  Set $\left| \mu_{K_1} \otimes \mathbb{Z}_p \right| = p^r$. The following are equivalent:
\begin{enumerate}
	\item $w_{1} \theta_{K_1/k_0, S^{\mathrm{min}}} (0)$ is in $p^r \mathbb{Z}[G]$
	\item As $\mathfrak{a}$ runs through classes in $\left( \Cl_{K_1} \otimes \mathbb{Z}_p \right)^{-}$, the elements $\varepsilon_{\mathfrak{a}, K_1/k_1, S_{k_1}^{\mathrm{ns}}}$ can be chosen so that their $p^r$th roots generate a cyclic extension of $k_0$ of degree $2p^{r+1}$, unramified away from $p$ and linearly disjoint over $k_0$ from the cyclotomic $\mathbb{Z}_p$-extension.
\end{enumerate} 
For conditions (i) and (ii) to hold, it is necessary that there exist $G$-isomorphisms
\begin{equation*}
	\left( \Cl_{K_0} \otimes \mathbb{Z}_p \right)^- \cong \left( \Cl_{K_1} \otimes \mathbb{Z}_p \right)^- \cong \mu_{K_1} \otimes \mathbb{Z}_p
\end{equation*}
\end{theorem1point3}

\begin{proof}
(i) $\rightarrow$ (ii):  Set $p^t = \left| \mathfrak{C}_0 \right|$.  Since $\mathfrak{K}_p$ is trivial, Proposition \ref{P:theta1divisibility} shows that the normalized $p$-valuations of the coefficients of $\theta_1$ are all $-1$.  Proposition \ref{P:theta0divisibility} shows that  $r=t$.  The necessary condition is then a consequence of Lemma 2.1 in \cite{bS2012b} and Proposition \ref{P:trivialproperties} (ii).

Now apply the result of Theorem \ref{T:bsharprank1}.  Choos an ideal $\mathfrak{a}$ representing a class of order $p^t$ in $\left( \Cl_{K_1} \otimes \mathbb{Z}_p \right)^-$, let $\omega(\mathfrak{a})$ be the element of $\mathfrak{W}_t$ from Lemma \ref{L:modifiedradical}.  Since the $p$-local Brumer-Stark conjecture for $K_1/k_0$ holds, 
\begin{equation*}
	\left( \sigma, \omega(\mathfrak{a}) \right)_{\Kum} \left[ \left( \sigma, \omega(\mathfrak{a}) 
	\right)_{N}^{p^{r-1}} \right]^{\frac{b_0+b_1}{b_0}} = 1
\end{equation*}
Condition (i) says precisely that $b_0 + b_1$ is $p$-integral, and thus the exponent on the right is in $p \mathbb{Z}_p$.  It follows that $\left( \sigma, \omega(\mathfrak{a}) \right)_{\Kum} = 1$.  By property (v) of the symbol $\left( \cdot, \cdot \right)_{\Kum}$, the field $L = K_1 (\sqrt[p^r]{\omega(\mathfrak{a})})$ is Abelian over $k_0$ (not just $K_0$, since $\omega(\mathfrak{a})$ is an anti-unit). 

As $\mathfrak{a}$ has order $p^t$ in $\left( \Cl_{K_1} \otimes \mathbb{Z}_p \right)^-$ and $\left(\omega(\mathfrak{a}) \right) = \mathfrak{a}^{p^t (1-\tau)}$, $\omega(\mathfrak{a})$ is not in $K_1^{\times p} \mu_{K_1}$. Thus, $\left[ L \colon K_1 \right] = p^r$ and $L$ is disjoint over $K_1$ from the cyclotomic $\mathbb{Z}_p$-extension of $K_1$.

If $\mathfrak{b}$ represents another class in $ \left( \Cl_{K_1} \otimes \mathbb{Z}_p \right)^-$, then there exists an exponent $c$ such that $\mathfrak{b}$ and $\mathfrak{a}^c$ differ by a factor of a principal ideal.  Therefore, $\omega(\mathfrak{b})$ can be chosen so that it differs from $\omega(\mathfrak{a})^c$ by an element of $K_1^{\times p^r}$.  It follows that the numbers $\omega(\mathfrak{a})$ as $\mathfrak{a}$ runs through the ideals representing classes in $A_1$ can all be chosen so that their $p$th roots lie inside of $L$.

Next, we prove the claim:
\begin{claim}
The Galois group of the extension $L/K_0$ is not isomorphic to $\mathbb{Z}/p^r \mathbb{Z} \times \mathbb{Z}/p \mathbb{Z}$.
\end{claim}

Accepting the claim for now, the implication (i) $\rightarrow$ (ii) follows immediately from Tate's formula  for $L_{K_1/k_1, S_{k_1}^{\mathrm{ns}}}(0)$ (Equation \eqref{E:quadL} in Proposition \ref{P:theta0divisibility}):
\begin{equation*}
	\mathfrak{a}^{w_1 \theta_{K_1/k_1, S_{k_1}^{\mathrm{ns}}}} = \mathfrak{a}^{2^{\left| S_{k_1}^{\mathrm{ns}} \right|-2} \left| \mathfrak{C}_1
	\right| (1-\tau)} = \left( \omega(\mathfrak{a})^{2^{\left| S_{k_1}^{\mathrm{ns}} \right|-2} \frac{\left| \mathfrak{C}_1 
	\right|}{p^t}} \right).
\end{equation*}
Choose $\varepsilon \left( \mathfrak{a}, K_1/k_1, S_{k_1}^{\mathrm{ns}} \right)$ to be the generator on the right.
The exponent on this generator is an integer relatively prime to $p$, so $\sqrt[p^r]{\omega(\mathfrak{a})}$ and $\sqrt[p^r]{\varepsilon \left( \mathfrak{a}, K_1/k_1, S_{k_1}^{\mathrm{ns}} \right)}$ generate the same extension of $K_1$.

\begin{proof}[Proof of Claim]
Assume, to reach a contradiction, that $\Gal(L/K_0)$ is isomorphic to $\mathbb{Z}/p^r \mathbb{Z} \times \mathbb{Z}/p \mathbb{Z}$. Since $L/K_1$ is cyclic, there is a unique subfield of $L$ of degree $p^{r-1}$ over $K_1$. By assumption, $\Gal (L/K_0)$ has multiple subgroups of order $p$, so there exists a field $F$ not containing $K_1$, Abelian over $k_0$,  with $K_0 \subset F \subset L$ and $\left[ F \colon K_0 \right] = p^r$. Since $F$ is abelian over $k_0$, there exists an anti-unit $\alpha$ in $\left( K_0^{\times} \right)^{1-\tau}$ for which $K_0 \left( \sqrt[\uproot{2} p^{r}]{\alpha} \right) = F$.  It follows that $K_1 \left( \sqrt[\uproot{2} p^r]{\alpha} \right) = L$, since otherwise, $\alpha$ would be a $p$th power in $K_1$, and then $K_1 = K_0 \left( \sqrt[p]{\alpha} \right)$ would be a subfield of $F$. 

There is thus an integer $c$ relatively prime to $p$ such that $\alpha$ differs from $\omega(\mathfrak{a})^c$ by a factor of a $p^r$th power. By definition, $\omega(\mathfrak{a})^c$ is the $p^r$th power of an ideal, so the ideal in $K_0$ generated by $\alpha$ can be written as $\mathfrak{b}^{p^r} \mathfrak{c}$ for some ideal $\mathfrak{c}$ in $K_0$ supported at places that ramify in $K_1/K_0$.

Because no prime ideal of $k_0$ splits in $K_0/k_0$ and ramifies in $k_1/k_0$, the ideal $\mathfrak{c}^{1-\tau}$ is trivial.  Applying $1-\tau$ to $(\alpha) = \mathfrak{b}^{p^r} \mathfrak{c}$ and recalling that $\alpha$ is in $\left( K_1^{\times} \right)^{(1-\tau)}$, we obtain $(\alpha^2) = \mathfrak{b}^{p^r (1-\tau)}$. If $\mathfrak{b}^{p^{r-1} (1-\tau)}$ were trivial in $\mathfrak{C}_0$, then applying $1-\tau$, we would find $\mathfrak{b}^{2p^{r-1} (1-\tau)} = \left( \beta^{1-\tau} \right)$ for some $\beta$ in $K_0^{\times}$.  Then $\alpha^4 = \beta^{p(1-\tau)} \zeta$ for some root of unity $\zeta$, so $K_1 (\sqrt[p]{\alpha^4})$ would be a degree $p$ extension of $K_1$ contained in both $L$ and the $p$-cyclotomic $\mathbb{Z}_p$-extension of $K_1$, a contradiction.  The equality $(\alpha^2) = \mathfrak{b}^{p^r(1-\tau)}$ thus shows the class of $\mathfrak{b}^{p^{r-1}(1-\tau)}$ has order $p$ in $\mathfrak{C}_0$.  

Lemma 3.5 in \cite{bS2012b} shows that the lift of $\mathfrak{b}^{p^{r-1}(1-\tau)}$ to $K_1$ is principal.  Applying $p (1-\tau)$, we find there is an element $\xi$ in $K_1$ such that $(\alpha^4) = (\xi^{p(1-\tau}))$, and hence $\omega(\mathfrak{a})^{4c}$ differs from $\xi^{p(1-\tau)}$ by an element of $K_1^{\times p^r} \mu_{K_1}$.  Thus, $K_1 \left( \sqrt[p]{\omega(\mathfrak{a})} \right)$ is a degree $p$ extension of $K_1$ contained in both $L$ and the $p$-cyclotomic extension of $K_1$, a contradiction.  We have now shown that $L/K_0$ is cyclic of order $p^{r+1}$. \noqed
\end{proof}

(ii) $\rightarrow$ (i):   Let $\mathfrak{a}$ be a prime ideal not dividing $p$ representing a class of order $p^t$ in $\mathfrak{C}_1$, so $\mathfrak{a}^{\tau} \neq \mathfrak{a}$.  Proposition \ref{P:theta0divisibility} shows that the power of $\mathfrak{a}$ dividing $\varepsilon_{\mathfrak{a}, K_1/k_1, S_{k_1}^{\mathrm{ns}}}$ is $2^{\left| S_{k_1}^{\mathrm{ns}} \right| -2} \left| \mathfrak{C}_1 \right|$.  As $\mathfrak{a}$ is unramified in the extension of $K_1$ generated by the $p^r$th root of $\varepsilon_{\mathfrak{a}, K_1/k_1, S_{k_1}^{\mathrm{ns}}}$, it follows that $p^r$ must divide $\left| \mathfrak{C}_1 \right|$.  Thus, $t \geq r$.

Let $L$ be the extension of $K_1$ in statement (ii).  Proposition \ref{P:theta0divisibility} and our assumption (ii) imply the generator $\omega(\mathfrak{a})$ of $\mathfrak{a}^{p^t(1-\tau)}$ can be chosen to make $\sqrt[\uproot{2} p^r]{\omega(\mathfrak{a})}$ generate $L/K_1$.  As $t \geq r$ and $L$ is Abelian over $K_0$, Property (v) of $(\cdot, \cdot)_{\Kum}$ shows $\left( \sigma, \omega(\mathfrak{a}) \right)_{\Kum} = 1$.  Theorem \ref{T:bsharprank1} then yields:
\begin{equation*}
	1 = 
	\begin{cases}
		\left( \sigma, \omega (\mathfrak{a}) \right)_{\N}^{ p^{t-1} \frac{b_0+b_1}{b_0}}, \quad &\text{if $t = r$.}\\
		\left( \sigma, \omega (\mathfrak{a}) \right)_{\N}^{ p^{t-1} \frac{b_1}{b_0}}, \quad &\text{if $t >
		r$.}\\
	\end{cases}
\end{equation*}
Lemma \ref{L:trivialnormsymbol} below will show that $\left( \sigma, \omega(\mathfrak{a}) \right)_{\N}^{p^{r-1}} \neq 1$.  The theorem follows immediately, since when the above exponents are considered in $\mathbb{Z}_p$, their $p$-valuations are $r-1$ except the first one when $b_0 + b_1$ is $p$-integral, i.e., when the coefficient of the identity map in $\theta_{K_1/k_0, S^{\mathrm{min}}}(0)$ is $p$-integral.  The theorem now follows from Proposition \ref{P:samedenominators}.
\end{proof}

\begin{lemma}\label{L:trivialnormsymbol}
Let $K_1/k_0$ be a TK extension.  Set  $\left| \mu_{K_1} \otimes \mathbb{Z}_p \right| = p^r$ and $\left| \mathfrak{C}_{1,p} \right| = p^t$, and assume that $t \geq r$. Let $\mathfrak{a}$ be an ideal representing a class of order $p^t$ in $\mathfrak{C}_1$, and let $\omega (\mathfrak{a})$ be the element of the modified radical group $\mathfrak{W}_t$ of $K_1/K_0$ from Lemma \ref{L:modifiedradical}.  If $\left( \sigma, \omega(\mathfrak{a}) \right)_{\Kum} = 1$, then $\left( \sigma, \omega(\mathfrak{a}) \right)_{\N}^{p^{r-1}} \neq 1$.
\end{lemma}

\begin{proof}
Assume, to the contrary, that $\left( \sigma, \omega(\mathfrak{a}) \right)_{\Kum} = \left( \sigma, \omega(\mathfrak{a}) \right)_{\N}^{p^{r-1}} = 1$.  Then we have $\omega(\mathfrak{a})^{\sigma-1} = \delta^{p^{t}}$ with $N_{K_1/K_0} \delta^{p^{r-1}} = 1$.  By Hilbert's Theorem 90, there exists $\gamma$ in $K_1^{\times}$ such that $\delta^{p^{r-1}} = \gamma^{\sigma-1}$.  Then $\omega(\mathfrak{a}) = \gamma^{p^{t-r+1}} \eta$ for some $\eta$ in $K_0^{\times}$.  Thus,
\begin{equation*}
	N_{K_1/K_0} (\omega(\mathfrak{a})) = N_{K_1/K_0} \left( \gamma^{p^{t-r+1}} \right) \eta^p, 
\end{equation*}
so $N_{K_1/K_0} (\omega(\mathfrak{a}))$ is a $p$th power in $K_0$.  But Proposition 2.2 in \cite{bS2012b} and the assumption that $K_1/k_0$ is a TK extension show that the norm map $N \colon \mathfrak{C}_{1} \rightarrow \mathfrak{C}_0$ is an isomorphism, thus $N_{K_1/K_0} \mathfrak{a}^{1-\tau}$ represents a class of order $p^t$ in $\mathfrak{C}_0$. We obtain a contradiction by observing then that $\left(N_{K_1/K_0} (\omega(\mathfrak{a})) \right) = N_{K_1/K_0} \mathfrak{a}^{p^t(1-\tau)}$ cannot be the $p$th power of a principal ideal.   
\end{proof}

The analogue of Theorem \ref{T:main} with $\varepsilon_{\mathfrak{a}, K_1/k_0, S^{\mathrm{min}}}$ replacing $\varepsilon_{\mathfrak{a}, K_1/k_1, S_{k_1}^{\mathrm{ns}}}$ can be proved without recourse to Theorem \ref{T:bsharprank1}:

\begin{theorem}\label{T:maintheoremalternate}
Let $K_1/k_0$ be a TK extension for which the $p$-local Brumer-Stark Conjecture holds.  Set $\left| \mu_{K_1} \otimes \mathbb{Z}_p \right| = p^r$. The following are equivalent:
\begin{enumerate}
	\item $w_{1} \theta_{K_1/k_0, S^{\mathrm{min}}} (0)$ is in $p^r \mathbb{Z}[G]$
	\setcounter{enumi}{2}
	\item As $\mathfrak{a}$ runs through classes in $\left( \Cl_{K_1} \otimes \mathbb{Z}_p \right)^{-}$, the elements $\varepsilon_{\mathfrak{a}, K_1/k_0, S^{\mathrm{min}}}$ can be chosen so that their $p^r$th roots generate a cyclic extension of $k_0$ of degree $2p^{r+1}$, unramified away from $p$ and linearly disjoint over $k_0$ from the cyclotomic $\mathbb{Z}_p$-extension.
\end{enumerate} 
\end{theorem}

\begin{proof}
(i) $\rightarrow$ (iii):  As we assumed that the $p$-primary Brumer-Stark conjecture for $K_1/k_0$, the extensions generated by the $p^r$th roots of the elements $\varepsilon_{\mathfrak{a}, K_1/k_0, S^{\mathrm{min}}}$ are Abelian over $k_0$.  Property (i) implies that they are unramified away from $p$. 

 Let $\mathfrak{a}$ represent a class of order $p^t$ in $\mathfrak{C}_1$.  We fill first see that $\varepsilon_{\mathfrak{a}, K_1/k_0, S^{\mathrm{min}}}$ is not in $K_1^{\times p} \mu_{K_1}$. Assume, to the contrary, that $\varepsilon_{\mathfrak{a}, K_1/k_0, S^{\mathrm{min}}}$ can be chosen in $K_1^{\times p}$.  Then\\ $N_{K_1/K_0} \left( \varepsilon_{\mathfrak{a}, K_1/k_0, S^{\mathrm{min}}} \right)$ is in  $K_0^{\times p}$.  Letting  $\varepsilon(\mathfrak{a})$ denote the element of $K_0$ associated with the ideal $N_{K_1/K_0} \mathfrak{a}$ by the Brumer-Stark conjecture for $K_0/k_0$, we find using Proposition \ref{P:theta0divisibility} and Lemma \ref{L:theta1properties} the equality of ideals in $K_0$
 \begin{equation*}
 	\left( N_{K_1/K_0} \left( \varepsilon_{\mathfrak{a}, K_1/k_0, S^{\mathrm{min}}} \right) \right) = \mathfrak{a}^{N_H w_1 \theta_0} 
	= N_{K_1/K_0} \left(\mathfrak{a} \right)^{q 2^{\left|S^{\mathrm{min}}\right|-2} \left| \mathfrak{C}_0 \right| (1-\tau)} \mathcal{O}_{K_1}.
\end{equation*}
Thus, in $K_0$ we have
\begin{equation*}
	\left( N_{K_1/K_0} \left( \varepsilon_{\mathfrak{a}, K_1/k_0, S^{\mathrm{min}}} \right) \right) 
	= N_{K_1/K_0} \left(\mathfrak{a} \right)^{q 2^{\left|S^{\mathrm{min}}\right|-2} \left| \mathfrak{C}_0 \right| (1-\tau)}.
\end{equation*}
Since $\mu_{K_1} \otimes \mathbb{Z}_p$ is not cohomologically trivial, $q$ is relatively prime to $p$. The surjectivity of the norm map $N \colon \mathfrak{C}_{1} \rightarrow \mathfrak{C}_0$ then shows the ideal on the right is not the $p$th power of a principal ideal, a contradiction.

We have now shown that $\varepsilon_{\mathfrak{a}, K_1/k_0, S^{\mathrm{min}}}$ can be chosen so that its $p^r$th roots generate an extension field $L$ of degree $p^r$ over $K_1$, Abelian over $k_0$,  disjoint over $K_1$ from the $p$-cyclotomic $\mathbb{Z}_p$-extension, with $L/K_1$ unramified away from $p$.  Property (i) of Proposition \ref{P:trivialproperties} shows that $L/k_0$ is unramified away from $p$.  

It remains to show 1) for every ideal $\mathfrak{b}$ representing a class in $\left( \Cl_{K_1} \otimes \mathbb{Z}_p \right)^-$, the element $\varepsilon_{\mathfrak{a}, K_1/k_0, S^{\mathrm{min}}}$ can be chosen so that its $p^r$th roots are in $L$, and 2) $L/K_0$ is cyclic.  The proof can be accomplished by mimicking the arguments of the claim and the paragraph preceding it in the proof of (i) $\rightarrow$ (ii) of Theorem \ref{T:main}.
\

(iii) $\rightarrow$ (i):  Let $H_0$ and $H_1$ be the unramified Abelian extensions of $K_0$ and $K_1$ corresponding to the groups $\mathfrak{C}_{0,p}$ and $\mathfrak{C}_{1,p}$ through class field theory.  The proof of Proposition 2.2 in \cite{bS2012b} shows that $H_1 = H_0 K_0$. Because $\mathfrak{K}_p$ is trivial, the proof also shows that the restriction map $\Gal (H_1/K_1) \rightarrow \Gal(H_0/K_0)$ is an isomorphism.  Let $g$ be an element of $\Gal (H_1/K_1)$ of order $p^t$.  Use Chebotarev's Density Theorem to choose a prime ideal $\mathfrak{q}$ in $K_0$ not dividing $p$ whose image under the Artin map for $H_1/K_0$ is $g$.  Because $g$ fixes $K_1$, $\mathfrak{q}$ splits completely in $K_1/K_0$. 

 Let $\mathfrak{Q}$ be a prime ideal of $K_1$ dividing $\mathfrak{q}$ and $\sigma_{\mathfrak{Q}}$ be the Frobenius automorphism of $\mathfrak{Q}$ in $\Gal (H_1/K_1)$.  The restriction of $\sigma_{\mathfrak{Q}}$ to $H_0$ is the Frobenius automorphism of $\mathfrak{q}$ in $\Gal (H_0/K_0)$.  This is also the restriction of the Frobenius automorphism of $\mathfrak{q}$ in $\Gal(H_1/K_0)$, i.e. $g$, to $H_0$.  Since $\sigma_{\mathfrak{Q}}$ and $g$ both fix $K_1$ and have the same restriction to $H_0$, they are identical.  Thus, $\mathfrak{Q}$ represents a class of order $p^t$ in $\mathfrak{C}_{1,p}$.
 
Because $\mathfrak{q}$ splits completely in $K_1/K_0$, the power of $\mathfrak{Q}$ dividing $\varepsilon_{\mathfrak{Q}, K_1/k_0, S^{\mathrm{min}}}$ is the coefficient of the identity map in $w_1 \theta_{K_1/k_0, S^{\mathrm{min}}} (0)$.  As $\mathfrak{Q}$ is unramified in the extension of $K_1$ generated by the $p^r$th root of $\varepsilon_{\mathfrak{Q}, K_1/k_0, S^{\mathrm{min}}}$ (recall that $\mathfrak{Q}$ does not divide $p$), $p^r$ must divide the coefficient of the identity in $w_1 \theta_{K_1/k_0, S^{\mathrm{min}}} (0)$.  Property (i) then follows from Proposition \ref{P:samedenominators}.
\end{proof}

We can strengthen the above theorems if we assume $K_1/k_0$ to be a TKNS extension.  First, we will use this assumption to bound the size of the maximal Abelian exponent $p^{r+1}$ extensions of $k_0$ and $k_1$ unramified away from $p$.

In the following lemma, for a finitely generated Abelian group $A$ with an action by complex conjugation, we let $\rk^+ (A)$ and $\rk^- (A)$ denote the $p$-ranks of $\left( A \otimes \mathbb{Z}_p \right)^+$ and  $\left( A \otimes \mathbb{Z}_p \right)^-$.
\begin{lemma}\label{L:reflection}
Let $K$ be a CM number field, and let $K^+$ be its maximal totally real subfield.  Assume that $K$ contains the $p$th roots of unity, and that no prime ideal of $K^+$ dividing $p$ splits in $K/K^+$.  Let $L/K^+$ be the maximal Abelian exponent $p$ extension of $K^+$ which is unramified away from $p$.  Then 
\begin{equation*}
	\left[ L : K^+ \right] = 
	p^{\rk^{-} \left( \Cl_K \right) + 1},
\end{equation*}
\end{lemma}

\begin{proof}
Let $T$ be the set of prime ideals of $K$ dividing $p$, and let $S$ be the empty set.  Let $\chi$ be the nontrivial character of $\Gal (K/K^+)$. The $S$-$T$ reflection theorem (see \cite[Theorem 5.4.5]{gG2003}) in this context states that
\begin{equation}\label{E:reflection}
	\rk^{+} \left( \Cl_{T}^{S} \right) - \rk^{-} \left( \Cl_{S}^{T} \right) = \rho_{\chi} (T,S).
\end{equation}
Here, $\Cl_{T}^{S} = \Cl_{T}$ is the projective limit of generalized class group corresponding through class field theory to the maximal Abelian $T$-ramified and $S$-split pro-$p$ extension of $K$, and similarly, $\Cl_{S}^{T} = \Cl^{T}$ is the generalized class group corresponding to the maximal Abelian $S$-ramified $T$-split pro-$p$ extension of $K$.  Finally, $\rho_{\chi} (T,S)$ in this setting is given by
\begin{equation*}
	\rho_{\chi} (T, S) = 1 +   \sum_{\substack{v \in T^+\\
									v \text{ split in } K}} 1,
\end{equation*}
where $T^+$ is the set of prime ideals of $K^+$ dividing $p$.  By assumption, this sum is empty, so $\rho_{\chi} (T, S) = 1$.

As no prime in $T^+$ splits in $K/K^+$, the prime ideals in $T$ generate an exponent $2$ subgroup of $\Cl^{-}$.    Therefore, $\rk^{-} \left( \Cl_{S}^{T} \right) = \rk^{-} \left( \Cl \right)$. Equation \eqref{E:reflection} thus becomes
\begin{equation*}
	\rk^{+} \left( \Cl_{T} \right) = \rk^{-} \left( \Cl \right) + 1.
\end{equation*}
The lemma now follows using the definition of $L$.
\end{proof}

\begin{corollary}\label{C:nosplitconditions}
Let $K_1/k_0$ be a TKNS extension for which the p-local Brumer-Stark conjecture holds.  Set $\left| \mu_{K_1} \otimes \mathbb{Z}_p \right| = p^r$. Conditions (i), (ii), and (iii) of Theorems \ref{T:main} and \ref{T:maintheoremalternate}  are equivalent to:

\begin{enumerate}
	\setcounter{enumi}{3}
	\item There exists a cyclic extension $L/k_0$ of degree $p^{r+1}$, unramified away 
	from $p$ and linearly disjoint over $k_0$ from the cyclotomic $\mathbb{Z}_p$-extension.
\end{enumerate}
In addition, when $(iv)$ holds, $L \left( \zeta_{p^{2r+1}} \right)^+$ is the maximal Abelian exponent $p^{r+1}$ extension of $k_0$ unramified away from $p$.
\end{corollary}

\begin{proof}
(ii) $\rightarrow$ (iv) is immediate.  

We will now show that (iv) $\rightarrow$ (i).  We first show that $L (\zeta_{p^{2r+1}})^+$ is the maximal Abelian exponent $p^{r+1}$ extension of $k_0$ unramified away from $p$.  Because $K_0 = k_0 (\zeta_{p^r})$, we know $L (\zeta_{p^{r}})$ is a cyclic extension of $K_0$ of degree $p^{r+1}$.  Since $L$ and the cyclotomic $\mathbb{Z}_p$-extension of $K_0$ are disjoint, the Galois group of $L (\zeta_{p^{2r+1}})$ over $K_0$ is isomorphic to $\mathbb{Z}/p^{r+1} \mathbb{Z} \times \mathbb{Z}/p^{r+1} \mathbb{Z}$.  Thus, $L(\zeta_{p^{2r+1}})^+$ is an Abelian exponent $p^{r+1}$ extension of $k_0$ unramified away from $p$. 

Let $M/K_0$ be the degree $p^2$ subextension of  $L (\zeta_{p^{2r+1}})/K_0$ for which $\Gal (M/K_0)$ is Abelian of exponent $p$;  $M$ is a CM field. By assumption, no prime ideal dividing $p$ splits in $K_0/k_0$, and Proposition \ref{P:trivialproperties} shows that $\left( \Cl_{K_0} \otimes \mathbb{Z}_p \right)^-$ is a nontrivial cyclic group. Thus, Lemma \ref{L:reflection} shows that $M^+$ is the maximal Abelian exponent $p$ extension of $k_0$ unramified away from $p$.    Then $L(\zeta_{p^{2r+1}})^+$ must be the maximal Abelian exponent $p^{r+1}$ extension of $k_0$ unramified away from $p$ since the Galois group over $k_0$ of any larger such field would have $p$-rank at least $3$, contradicting the maximality of $M^+$.

To prove (ii), we begin by observing that Proposition \ref{P:trivialproperties} shows $k_1/k_0$ is an Abelian exponent $p$ extension of $k_0$ which is unramified away from $p$.  Therefore, $k_1$ is contained in $M$, hence in $L(\zeta_{p^{2r+1}})$, and thus $K_1$ is also contained in $L(\zeta_{p^{2r+1}})$. 

The subgroup of $\Gal \left(L (\zeta_{p^{2r+1}})/K_0 \right)$ fixing $K_1$ has index $p$.  Let $H$ be a sugroup of $\mathbb{Z}/p^{r+1} \mathbb{Z} \times \mathbb{Z}/p^{r+1} \mathbb{Z}$ of index $p$.  Choose an element $h$ of $H$ not contained in $p \mathbb{Z}/p^{r+1}\mathbb{Z} \times p \mathbb{Z}/p^{r+1} \mathbb{Z}$.  Then $h$ generates a cyclic subgroup of $H$ of order $p^{r+1}$.  The quotient $\left( \mathbb{Z}/p^{r+1} \mathbb{Z} \times \mathbb{Z}/p^{r+1} \mathbb{Z} \right) / \left< h \right>$ is cyclic of order $p^{r+1}$. Thus, there is a field $F$ with 
\begin{equation*}
	k_0 \subset K_1 \subset F \subset L(\zeta_{p^{2r+1}})
\end{equation*}
for which $F/k_0$ cyclic of degree $2p^{r+1}$.  This cyclicity implies that $F$ is disjoint over $k_0$ from the cyclotomic $\mathbb{Z}_p$-extension.

Let $\gamma$ be an element of $K_1^{\times}$ such that $F = K_1 \left( \sqrt[\uproot{2} p^r]{\gamma} \right)$. We can choose $\gamma$ to be in $(K_1^{\times})^{1-\tau}$ since $F/k_1$ is Abelian.  Because $F/K_1$ is unramified away from $p$, we can write
\begin{equation*}
	\left( \gamma \right) = \mathfrak{a}^{p^r} \prod_{\mathfrak{P}_i \mid p} \mathfrak{P}_i^{e_i},
\end{equation*}
where the prime ideals $\mathfrak{P}_i$ ramify in $K_1/K_0$, hence divide $p$. As no prime ideal dividing $p$ splits in $K_1/k_1$, it follows that $\mathfrak{a}^{2 p^r}$ is trivial in $\mathfrak{C}_{1}$ and $\left( \gamma^{1-\tau} \right) = \mathfrak{a}^{p^r(1-\tau)}$.

Set $p^t = \left| \mathfrak{C}_{1,p} \right|$.  If $\mathfrak{a}^{2p^{r-1}}$ is trivial in $\mathfrak{C}_1$, then $\mathfrak{a}^{2p^{r-1} (1-\tau)} = (\eta^{1-\tau})$ for some $\eta$ in $K_1^{\times}$. Then $\gamma^{2(1-\tau)}$ differs from $\eta^{p(1-\tau)}$ by a factor of a root of unity.  But we have $F = K_1 \left( \sqrt[\uproot{2} p^r]{\gamma^4} \right) = K_1 \left( \sqrt[\uproot{2} p^r]{\gamma^{2(1-\tau)}} \right)$. We have a contradiction, since  $F$ is disjoint over $K_1$ from the cyclotomic $\mathbb{Z}_p$-extension.  Therefore, the class of $\mathfrak{a}$ in $\mathfrak{C}_1$ has order $p^r$ or $2p^r$, so $t \geq r$ and there exists an ideal $\mathfrak{b}$ representing a class in $\mathfrak{C}_{1,p}$ such that $\mathfrak{b}^{p^{t-r}}$ is in the same class as $\mathfrak{a}^2$. 

Write $\mathfrak{b}^{p^{t-r} (1-\tau)} = (\xi^{1-\tau}) \mathfrak{a}^{2(1-\tau)}$ for some $\xi$ in $K_1^{\times}$.  If $\omega(\mathfrak{b})$ is as in Lemma \ref{L:modifiedradical}, then $\left(\omega(\mathfrak{b})\right) = \left( \xi^{p^r(1-\tau)} \gamma^{2(1-\tau)} \right)$. We can thus choose $\omega(\mathfrak{b}) = \xi^{p^r(1-\tau)} \gamma^{2(1-\tau)}$, and then $K_1 \left( \sqrt[\uproot{2} p^r]{\omega(\mathfrak{b})} \right) = F$.  As $F/K_0$ is Abelian, $\left( \sigma, \omega(\mathfrak{b}) \right)_{\Kum} = 1$.  It now follows as in the proof of (ii) $\rightarrow$ (i) in the Main Theorem \ref{T:main} that $\theta$ has $p$-integral coefficients.
\end{proof}

\begin{corollary}\label{C:Leopoldt}
Let $K_1/k_0$ be a TKNS extension.  Then a sufficient condition for Leopoldt's conjecture for $k_0$ is $\left| \left( \Cl_{K_0} \otimes \mathbb{Z}_p \right)^- \right| \neq \left| \mu_{K_0} \otimes \mathbb{Z}_p \right|$.
\end{corollary}

\begin{proof}
Property (ii) of Proposition \ref{P:trivialproperties}  shows that $\left( \Cl_{K_0} \otimes \mathbb{Z}_p \right)^-$ is a cyclic group. If Leopoldt's conjecture for $k_0$ is false, then Lemma \ref{L:reflection} shows that $k_1/k_0$ is the first layer of a $\mathbb{Z}_{p}$ extension of $k_0$.  Condition (iv) of Corollary \ref{C:nosplitconditions} holds, and then Theorem \ref{T:main} shows that $\left( \Cl_{K_0} \otimes \mathbb{Z}_p \right)^- \cong \mu_{K_0} \otimes \mathbb{Z}_p$.
\end{proof}

\section{Examples}\label{S:examples}

In this section, we provide a simple method for computationally producing examples of TKNS extensions.  The method will be discussed only for sextic extensions of real quadratic base fields, although it can be adapted to produce examples with higher degree extensions of higher degree base fields.

First, use a computer to search through real quadratic fields with discriminants in a predetermined set, identifying each field $k_0$ whose extension $K_0 = k_0 (\sqrt{-3}) $ has the following properties:
\begin{itemize}
	\item No place dividing $3$ splits in $K_0 /k_0$,
	\item $K_0$ has nontrivial cyclic $3$-primary class group.
\end{itemize} 
	
By Lemma \ref{L:reflection}, each such field $k_0$ must have degree $3$ Abelian extensions unramified away from $3$ other than the cyclotomic one.  Compute such an extension field and call it $k_1$.  Compute the class number of $K_1 = k_1 (\sqrt{-3})$.  If the class numbers of $K_0$ and $K_1$ are exactly divisible by the same power of $3$, then the norm map $N \colon \left( \Cl_{K_1} \otimes \mathbb{Z}_p \right)^- \rightarrow \left( \Cl_{K_0} \otimes \mathbb{Z}_p \right)^-$ has trivial kernel.  Then $K_1/k_0$ is a TKNS extension.

If the Brumer-Stark conjecture for $K_1/k_0$ holds, then $K_1/k_0$ will satisfy all of the hypotheses of Theorem \ref{T:main} and Corollary \ref{C:nosplitconditions}.  Although the Brumer-Stark conjecture is not generally known to hold for sextic extensions of real quadratic base fields, it was computationally verified for real quadratic base fields with small discriminant in Greither-Roblot-Tangedal \cite{GRT2004}.  

The table below was produced by implementing the above method in the software PARI/GP. Each row corresponds to a TKNS extension $K_1/k_0$ of degree $6$ with base field $k_0 = \mathbb{Q}(\sqrt{n})$.  The columns provide $n$, the cardinality $h_3$ of the cyclic $3$-primary class group of $K_0$, the minimal polynomial $p$ for a generator of the cubic extension $k_1/k_0$, a flag indicating if the module $\mathfrak{K}_p$ is trivial, the value $w_1 \theta_{K_1/k_0, S^{\mathrm{min}}}(0) = 6 \theta_{K_1/k_0, S^{\mathrm{min}}}(0)$, and a flag indicating if condition (iv) from Corollary \ref{C:nosplitconditions} holds.  Those rows where $\mathfrak{K}_p$ is trivial correspond to extensions $K_1/k_0$ satisfying all of the hypotheses of Theorem \ref{T:main} and Corollary \ref{C:nosplitconditions}.  The Brumer-Stark conjecture is known for extensions in the table by \cite{GRT2004}.

We now examine the first two extensions in the table in more detail.

Let $k_0 = \mathbb{Q}(\sqrt{29})$.  Let $k_1$ be the cubic extension of $k_0$ generated by a root of $p = x^3 - 6 x - \sqrt{29}$.  Let $K_0 = k_0 \left( \sqrt{-3} \right)$ and $K_1 = k_1 \left( \sqrt{-3} \right)$.  Computations performed with PARI/GP indicate there is a cyclic degree $9$ extension of $k_0$, unramified away from $3$ and linearly disjoint over $k_0$ from the $\mathbb{Z}_p$-extension.  According to Corollary \ref{C:nosplitconditions}, $w_1 \theta_{K_1/k_0, S^{\mathrm{min}}} (0)$ must be in $3 \mathbb{Z}[G]$.  This is confirmed by PARI/GP, which shows that
\begin{equation*}
	w_1 \theta_{K_1/k_0, S^{\mathrm{min}}} (0) = ( 18 -  6 \sigma  -  6 \sigma^2) (1 - \tau),
\end{equation*}
where $\sigma$ and $\tau$ are elements of $G = \Gal (K_1/k_0)$ of orders $3$ and $2$ respectively.

Now let $k_0 = \mathbb{Q}(\sqrt{43})$.  Let $k_1$ be the cubic extension of $k_0$ generated by a root of $p = x^3-21x+2\sqrt{43}$.  Let $K_0 = k_0 \left( \sqrt{-3} \right)$ and $K_1 = k_1 \left( \sqrt{-3} \right)$.  Computations performed with PARI/GP indicate that there is only one cyclic degree 9 extension of $k_0$ unramified away from 3 --- the cyclotomic one.  According to Corollary \ref{C:nosplitconditions}, $w_1 \theta_{K_1/k_0, S^{\mathrm{min}}} (0)$ should be in $\mathbb{Z}[G]$ but not in $3 \mathbb{Z}[G]$.  This is confirmed by PARI/GP, which shows that
\begin{equation*}
	w_1 \theta_{K_1/k_0, S^{\mathrm{min}}} (0) = ( 20 -  4 \sigma  -  4 \sigma^2) (1 - \tau),
\end{equation*}
where $\sigma$ and $\tau$ are elements of $G = \Gal (K_1/k_0)$ of orders $3$ and $2$ respectively.

Finally, we single out the row of the table with $d=173$ as an extension satisfying the sufficient condition for the $p$-primary Leopoldt's conjecture in Corollary \ref{C:Leopoldt} (although the conjecture is already trivial in this example because $k_0$ is a real quadratic field).

\newpage

\begin{center}
\renewcommand{\arraystretch}{1.3}
\begin{tabular}{|c | c |>{\small}c | c|>{\small}c | c |}
	\hline
	$n$ & $h_3$ & $p$ & \small $\mathfrak{K}_p$ trivial & $w_{1} \theta (0)$ & \small (iv) holds \\
	\hline
	29 & 3 &  $x^3 - 6 x - \sqrt{29}$ & yes &  $\left( 18 - 6 \sigma - 6 \sigma^2 \right) (1-\tau)$ & yes\\
	\hline
	43 & 3 &   $x^3 - 21x - 2\sqrt{43}$ & yes &  $\left( 20 - 4 \sigma - 4 \sigma^2 \right) (1-\tau)$ & no\\
	\hline
	58 & 3 &  $x^3 - 30x - 4 \sqrt{58}$ & yes &  $\left( 20 - 4 \sigma - 4 \sigma^2 \right) (1-\tau)$ & no \\
	\hline
	62 & 3 &   $x^3 - (27 + 3 \sqrt{62})x - (65 + 8 \sqrt{62})$ & no & &\\
	\hline
	67 & 3 &   $x^3 - 21 x - 4\sqrt{67}$ & yes & $\left( 36 - 12 \sigma - 12 \sigma^2 \right)(1-\tau)$ & yes\\
	\hline
	74 & 3 &   $x^3 - 21 x - (11 + 3 \sqrt{74})$ & yes & $\left( 24 - 24 \sigma + 12 \sigma^2 \right)(1-\tau)$ & yes\\
	\hline
	77 & 3 &  $2x^3 - 24 x - (5 + 3 \sqrt{77})$ & no & &\\
	\hline
	79 & 3 &  $x^3 - 21 x - (9 + 2 \sqrt{79})$ & yes & $\left( 24 + 24 \sigma - 36 \sigma^2 \right)(1-\tau)$ & yes\\
	\hline
	82 & 3 &  $x^3 - (11+\sqrt{82}) x - 1$ & yes & $ \left( 8-16\sigma + 20 \sigma^2 \right) (1-\tau)$ & no\\
	\hline
	83 & 3 &  $x^3 - 30 x - (2 + 6 \sqrt{83})$ & no & &\\
	\hline
	85 & 3 &  $2 x^3 - (13+\sqrt{85}) x - 2$ & yes & $ \left( -4 + 8 \sigma + 8 \sigma^2 \right)(1-\tau)$ & no\\
	\hline
	93 & 3 &  $x^3 - 24 x - 8$ & yes & $ \left( 10 - 2 \sigma - 2 \sigma^2 \right) (1-\tau)$ &  no\\
	\hline
	103 & 3 & $x^3 - (14 + \sqrt{103}) x - (23 + 2\sqrt{103})$ & no & &\\
	\hline
	106 & 3 & $x^3 - (17 + \sqrt{106}) x - (13 + \sqrt{106})$  & yes & $ \left( 20 +20 \sigma - 28 \sigma^2 \right) (1-\tau)$ & no\\
	\hline
	109 & 3 & $x^3 - 12 x - \sqrt{109}$ & yes & $ \left( 44 - 16 \sigma - 16 \sigma^2 \right) (1-\tau)$ & no\\
	\hline
	113 & 3 & $x^3 - 15 x - 2 \sqrt{113} $ & yes & $ \left( 10 - 2 \sigma - 2 \sigma^2 \right) (1 -\tau)$ & no \\
	\hline
	122 & 3 & $x^3 - 15 x - 2 \sqrt{122}$ & yes & $ \left( 36 - 12 \sigma - 12 \sigma^2 \right) (1-\tau)$ & yes\\
	\hline
	131 & 3 & $x^3 - 33 x - 2 \sqrt{131}$ & yes & $ \left( 20 - 4 \sigma - 4 \sigma^2 \right) (1-\tau)$ & no\\
	\hline
	137 & 3 & $x^3 - 15 x - \sqrt{137}$ & yes &  $ \left( 18 - 6 \sigma - 6 \sigma^2 \right) (1-\tau)$ & yes \\
	\hline
	139 & 3 & $x^3 - (14 + \sqrt{139}) x - (23 + 2 \sqrt{139})$ & no & &\\
	\hline
	142 & 3 & $x^3 - 30 x - (36 + 2 \sqrt{142})$ & yes & $\left( 28 + 40 \sigma - 44 \sigma^2 \right) (1-\tau)$ & no\\
	\hline
	151 & 3 & $x^3 - 21 x - (24 + \sqrt{151})$ & no & &\\
	\hline
	173 & 9 & $2x^3 - (45 + 3 \sqrt{173}) x - (56 + 4 \sqrt{173})$ & yes & $ \left( 38 - 10 \sigma - 10 \sigma^2 \right) (1-\tau)$ & no\\	
	\hline
	179 & 3 & $x^3 - 39 x - (7 + 6 \sqrt{179})$ & no & &\\
	\hline
	181 & 3 & $x^3 - 12 x + \sqrt{181}$ & no & & \\
	\hline
	182 & 3 & $x^3 - (45 + 3 \sqrt{182}) x - (55 + 4 \sqrt{182})$ & yes & $\left( 36 - 36 \sigma +24 \sigma^2 \right) (1-\tau)$ & yes\\
	\hline
	183 & 3 & $x^3 - 45 x - 6 \sqrt{183}$ & yes &  $ \left( 12 + 0 \sigma + 0 \sigma^2 \right) (1-\tau)$ & yes\\
	\hline
	199 & 3 & $x^3 - (20 + \sqrt{199}) x - (33 + 2 \sqrt{199})$ & no & &
	 \rule[-8pt]{0pt}{5pt}\\
	\hline
\end{tabular}
\end{center}

\end{document}